\documentclass[12pt,leqno]{amsart}
\usepackage{amssymb,amsmath,amsthm}
\usepackage{graphicx}
\usepackage{color}
\usepackage[utopia]{mathdesign}

\definecolor{refkey}{gray}{.75}
\def\e{{\rm e}}
\def\eps{\varepsilon}

\def\d{{\rm d}}

\def\div{{\mathrm{div}\;}}

\def\supp{{\mathrm{supp}\, }}

\def\ddt{\frac{\d}{\d t}}

\def\R {\mathbb{R}}

\def\V {{\mathcal V}}

\def\C {{\mathcal C}}
\def\D {{\mathcal D}}

\def\f{\mathbf{f}\;\,}

\def\T {{\mathbb T}}

\def\Q {{\mathcal Q}}

\def\W {{\mathcal W}}

\def \l {\langle}
\def \r {\rangle}
\def \pt {\partial_t}
\def\curl{\mathrm{curl}\,}

\def \and{\quad\text{and}\quad}

\newcommand{\cic}[1]{\mathbf{#1}}
\def\u{\cic{u}}
\def\w{\cic{w}\;}
\def\v{\cic{v}}

\def \no#1#2#3 {{\bf #1} (#3), #2.}
\def \eds#1#2#3 {#1, #2, #3.}

\newtheorem{proposition}{Proposition}[section]

\newtheorem{theorem}{Theorem}[section]

\newtheorem{lemma}{Lemma}[section]
\theoremstyle{definition}
\newtheorem{definition}[proposition]{Definition}

\newtheorem{remark}{Remark}[section]
\newtheorem*{remark*}{Remark}
\newtheorem*{warn*}{A word of warning}

\numberwithin{equation}{section}

\title[Euler equations on planar nonsmooth domains]
{The Euler equations in planar nonsmooth convex domains}
\author[C. Bardos]{Claude Bardos}
\address{Universit{\'e} Paris-Diderot, Laboratoire J.-L. Lions, BP 187,
\newline\indent
75252 Paris Cedex 05, France}
\email{claude.bardos@gmail.com {\rm (C.\ Bardos)}}
\author[F.\ Di Plinio]{Francesco Di Plinio}
\address{INdAM - Cofund Marie Curie Fellow at Dipartimento di Matematica, \newline \indent Universit\`a degli Studi di Roma ``Tor Vergata'', \newline  \indent Via della Ricerca Scientifica,   00133 Roma,  Italy   \newline \indent \centerline{and}   \indent
Dept.\ of Mathematics  \ \& \ The Institute for Scientific Computing and Applied Mathematics,
\newline\indent
Indiana University
\newline\indent
831 East Third Street, Bloomington, Indiana  47405, U.S.A. }
\email{fradipli@indiana.edu {\rm (F.\ Di Plinio)} }
\author[R.\ Temam]{Roger Temam}
\address{The Institute for Scientific Computing and Applied Mathematics, \newline\indent
Indiana University
\newline\indent
831 East Third Street, Bloomington, Indiana  47405, U.S.A.}
\email{temam@indiana.edu {\rm (R.\ Temam)} }

\date{\today}
\begin{document}
\begin{abstract}
  As a model problem   for the
barotropic mode of the primitive equations of the oceans and atmosphere,  we consider the   Euler system on a bounded convex planar domain $\Omega$, endowed 
with non-penetrating boundary conditions. For $\frac43 \leq p \leq 2$, and   initial and forcing data  with $L^p(\Omega)$ vorticity
we show the existence of a weak solution, enriching and extending the results of Taylor \cite{Tay}. 

In the physical  case of a rectangular domain $\Omega=[0,L_1] \times[0,L_2]$,   a similar result holds for all $2<p<\infty$ as well. Moreover, by means of   a new   $\mathrm{BMO}$-type  regularity estimate for the Dirichlet problem on a planar domain with corners, we prove uniqueness of solutions with bounded initial vorticity.  
\end{abstract}
\maketitle



{\footnotesize  \noindent In print on
    \emph{Journal of Mathematical Analysis and Applications}, Volume 407, Issue 1, Nov 2013, Pages 69 to 89
    \vskip0.5mm \noindent
    \verb|http://dx.doi.org/10.1016/j.jmaa.2013.05.005|
    \vskip0.5mm \noindent
    Received 1 Dec  2012, available online 9 May 2013.
    Submitted by Pierre Lemarie-Rieusset}
\section{Introduction}\label{s1}

Let $\Omega \subset \R^2$ be a bounded open set.  We are concerned with the Euler equations,    describing the motion of a perfect inviscid fluid inside $\Omega$,
\begin{equation} \label{P}  \tag{P}
\begin{cases}
\pt \cic{u}(x,t)  + (\u \cdot \nabla) \cic{u}(x,t) + \nabla \pi(x,t) = {\f}(x,t), & x\in \Omega, t \in (0,T), \\
\nabla \cdot \cic{u} (x,t) =0, & x\in \Omega, t \in (0,T), \\
\end{cases}
\end{equation}
where $\u(x,t)=\big(u_1(x_1,x_2,t), u_2(x_1,x_2,t)\big)$ stands for the velocity of the fluid particle located at $x$ at time $t$,  $\nabla \pi (x,t)$ stands for the pressure gradient, and ${\f}={\f}(x,t)$ is an external forcing term. We endow \eqref{P} with initial and boundary conditions
\begin{align}
&\u(x,0)= \u_0(x), \qquad\; \; \;\;x \in \Omega, \label{P-IC} \\
& \u(x,t) \cdot \mathbf{n}(x)= 0, \qquad  x \in  \partial \Omega, \, t \in (0,T) \label{P-BC};
\end{align}
we remark that the impermeability boundary condition \eqref{P-BC} has to be properly reformulated when $\partial \Omega$ is not regular enough to admit a normal vector $\mathbf{n}$ almost everywhere.

This work is primarily motivated by the study of the well-posedness of the barotropic mode of the inviscid primitive equations of the atmosphere and the oceans
\cite{RTTHandbookJMPA}.  As explained in  \cite{CST10,CSTT12,RTTHandbookJMPA}, \ a certain vertical modal expansion of the primitive equations leads to an infinite system of coupled barotropic - baroclinic modes.  In a first approximation, one can neglect the baroclinic modes and we obtain for the barotropic mode a system of equations very similar to the two-dimensional inviscid Euler equations in a rectangle.  Henceforth this system is called the barotropic system. The study of the well-posedness of the barotropic system is thus very similar to the study of the well-posedness of the (inviscid) incompressible Euler equations in a rectangle $\Omega.$ {In this article we consider the particular case of exactly the Euler system; the more general case will be considered elsewhere.}

The study of the well-posedness of the Euler equations of incompressible fluid has a long history regarding both weak and strong solutions, starting with \cite{LL,W}, the latter article of Wolibner containing the first proof of global in time existence and uniqueness of regular solutions in two dimensions. This result was simplified and extended by Kato in \cite{Kato};  when $\Omega$ is sufficiently smooth, with some additional work, one obtains up to $\mathcal{C}^\infty$ solutions, see e.\ g.\ \cite{Tem75,Tem86}, {or even analytic solutions \cite{Bardos2}.}

The notion of weak solution to the two-dimensional Euler equations has been introduced by Yudovich \cite{Yudo1}. Yudovich, and later Bardos with a different (vanishing viscosity) approach \cite{Bardos} consider initial data with $L^p$ vorticity, and show existence of weak solutions and, in the case of bounded initial vorticity, uniqueness. Among many references on the well-posedness of the incompressible Euler equations in a bounded smooth domain,  let us quote the classical articles \cite{Bardos,BM,Yudo2} (see also \cite{Tem75,Tem86}).
 Other weaker notions of solutions (e.\ g.\ initial datum in $\cic{L}^2(\Omega)$, with no assumption on the initial vorticity) have been considered by several   authors. See for instance  \cite{DipL,Sch1993,Shn03} and Remark  \ref{remturb} below.
 
{With the motivations indicated above, our aim in this article is twofold. Firstly, for general bounded convex domains $\Omega$, and for divergence-free initial datum $\cic{u}_0$ with $L^p(\Omega)$ vorticity, we prove the existence of solutions of \eqref{P} in $L^\infty(0,T; \cic{W}^{1,p}(\Omega))$, in the range $\frac43 \leq p \leq 2$.   
 Our proof follows the general scheme first devised in \cite{Tay}} {(which considers the range $\frac{12}{7}<p\leq2$), developing several points not explicitly addressed therein}: a weak solution is obtained as the limit of smooth solutions to \eqref{P} set on smooth convex subdomains $\Omega_n$ increasing to $\Omega$. {We remark that this approach has also been used in \cite{GVA} to construct (possibly weaker notions of) Euler solutions in more general domains  (complements of a finite number of compact connected sets with positive Sobolev capacity); however, for the class of domains under our consideration, our proof is simpler than the one in \cite{GVA}.}

{Proceeding as in \cite{Tay}, we make a more systematic study of the approximated problems $($P$_n)$.  In particular, we impose a uniform Lipschitz character to the domains $\Omega_n$. This uniformity reflects on the uniform boundedness of the Leray projectors associated to each subdomain, which we exploit in our compactness arguments. As a result, our   solution is slightly more regular than the one constructed in \cite{Tay},   being continuous in time with values in $\cic{L}^2(\Omega)$ and belonging to $\cic{W}^{1,p}(\Omega) $ for each time $t \in [0,T]$, and satisfying \eqref{P}   almost everywhere in $\Omega \times (0,T)$.}
 {The range of exponents $1 <p <\frac43$ can also be dealt with, by working with a weaker notion of solution than the one given in Subsection \ref{DefWS}, as for example in \cite{GVA}}. However, the restriction $1<p\leq2$ corresponds to   the sharp range of exponents for the  regularity of the Biot-Savart law (see Theorem \ref{ercvx}), giving the gradient of the velocity in terms of the vorticity, in a general convex domain, and seems unavoidable if one aims for a \emph{reversible} existence result (in the sense that $\u(t)$ for $t>0$ belongs to the same space  where the data $\u(0)$ is required to be, so that we can solve the backward Euler equations with initial data $\u(t)).$  The proof of existence of solutions is the object of   Proposition \ref{ThEx} and Theorem \ref{ThEx1}.

 {Our second aim in the article is to consider domains $\Omega$ of specific interest to us, in particular the rectangle $\Omega=[0,L_1]\times[0,L_2]$.
In this case we are able, in Theorem \ref{ThEx3}, to extend the existence result of Theorem \ref{ThEx1} to the range $p\in (2, \infty]$. Furthermore, and this is the main object of Theorem \ref{ThEx3}, we prove   uniqueness of solutions in the case $p=\infty$.}
To the best of our knowledge, this is the first uniqueness result for Euler solutions on a domain  with corners without requiring additional assumptions on the initial vorticity other than being bounded. In \cite{Lac}, the author proves a uniqueness result for a class of domains somehow complementary to ours, that is domains with a finite number of corners having angles greater than $\pi/2$, assuming that the initial vorticity is bounded and has definite signum.  Our proof follows Yudovich's energy method, and relies on the endpoint $L^\infty(\Omega) \to \cic{W}^{2,\mathrm{bmo}}(\Omega)$ regularity result   for the solution to the Dirichlet problem on a rectangle, Proposition \ref{bmolaw}, which appears to be new, to the best of our knowledge.  We do not dwell on the latter point, but unbounded initial vorticities with $\log\log$-type blowup of the $L^p$-norms as $p\to \infty$, like in \cite{Yudo2}, would also suffice for uniqueness.

Theorem \ref{ThEx3} holds \emph{verbatim}  for a more general class of domains,  that is bounded domains with piecewise smooth boundary and with corners  of aperture of the form $\alpha=\frac\pi m$, for some integer $m\geq 2.$ We briefly discuss this extension in Subsection \ref{remarkpoly}.
 
  \subsection*{Plan of the paper.} In Section \ref{s2}, we develop the necessary tools for the analysis of \eqref{P} on a bounded convex non-smooth domain $\Omega$. Section \ref{Sec3} contains the endpoint-type regularity result for the Dirichlet problem on a rectangle, which will be instrumental in establishing uniqueness of solutions. In Section \ref{s3} we give a weak formulation of \eqref{P}, and construct a weak solution to \eqref{P} on a bounded convex domain (Proposition \ref{ThEx}). Section \ref{s5} contains the statements  and the proofs of the main results, Theorems \ref{ThEx1} and \ref{ThEx3}.   In Section \ref{s6}, we make some additional remarks: in particular,   we briefly outline, for comparison's sake, the analogue of Theorem  \ref{ThEx3} in the space-periodic case $\Omega = \mathbb T^2. $ We also discuss some extensions of Theorem \ref{ThEx3} to a more general class of domains with corners.

\subsection*{Notation.} Given a domain $\Omega \subset \R^2$, and  scalar functions $u,v:\Omega \to \R$,  vector valued functions $\u,\v:\Omega \to \R^2$, we denote
$$
(u,v)_{\Omega} := \int_{\Omega} u   v,\qquad (\u,\v)_{\Omega} := \int_{\Omega} \u \cdot \v.
$$
Throughout, \label{nota} for $p \in [1,\infty]$, we use the notations
$$
p'= \frac{p}{p-1}, \qquad p^*= \begin{cases} \frac{2p}{2-p} & 1\leq p <2 \\ \infty &p\geq 2\end{cases}
$$ respectively for the   H\"older and Sobolev conjugate exponents of $p$.
 
We set up our notation for bump functions. Let $d\geq 1,$ and $\phi^d:\R^d\to \R$ be a smooth nonnegative radial function    supported in $ {\{|x|\leq \frac12\}}$ and with $\int_{\R^d} \phi^d=1$. We will make use of the $L^1$-normalized bump functions$$\phi^d_\eps(x) = (\mathrm{Dil}^1_\eps \phi^d\big)(x):=  \eps^{-d}\phi\big(\eps^{-1}x\big), \qquad \eps>0. $$ 

\subsubsection*{Acknowledgments.}
This work was supported in part by NSF Grants DMS 0906440, and DMS 1206438 and by the Research Fund of Indiana University.  
The authors acknowledge very useful discussions on the subject  with Madalina Petcu  and wish to thank Vlad Vicol for bringing to their attention very
useful references. Furthermore, the authors are grateful to the anonymous referee for his/her valuable comments to the first draft of this article, leading to an improvement of the final presentation.

\section{Elliptic regularity in a bounded convex domain} \label{s2}
In this section, we set the foundation for our analysis of the Euler system \eqref{P}. Throughout the section,  $\Omega $ is an open, bounded, convex   subset of $\R^2$ which contains the origin. In Section \ref{Sec3}, we will specialize to the case of a rectangular domain and develop further elliptic regularity results.

 We first recall the analytic properties of the boundary $\partial\Omega$ and  construct an approximation of
$\Omega$ by an increasing sequence of convex smooth subdomains with  uniformly Lipschitz boundary.
Then, we describe the normal trace operator on $\Omega$, introduce the class of tangential vector fields, and establish the Helmholtz decomposition of    $\mathbf{L}^p(\Omega)$, for $1<p<\infty$.
Finally, we discuss some   regularity results for the Dirichlet problem in $\Omega$, which we exploit to define the spaces in which the evolution of the Euler system \eqref{P} will take place.  \subsection{Regularity and approximation of bounded convex domains}
\noindent
We begin with a proposition.
\begin{proposition} \label{cvx1} Let $\Omega$ be a bounded convex open set containing the origin. There exist positive constants $M_\Omega, \delta_\Omega, $ and a finite collection of open squares $\{Q_i: i=1, \ldots, N_\Omega\}$ of diameter $\delta$ such that:
\begin{itemize}
\item[$\cdot$] $\displaystyle \overline{\Omega} \subset \bigcup_{i=1}^{N_\Omega} Q_i, \quad Q_i = c(Q_i) + Q_0,$   $\quad Q_0=\{ |y_1|, |y_2| < \delta_\Omega,  \},$
\item[$\cdot$] whenever $ Q_i \cap \partial \Omega\neq \emptyset   $, there exists a function  $\beta_i:(-\delta_\Omega,\delta_\Omega) \to \R$, in the coordinates with origin the center of $Q_i$ and oriented along the sides of $Q_i$, with the properties
\begin{itemize}
\item[$\cdot$] $\beta_i$ is convex and Lipschitz with constant $M_\Omega$,
\item[$\cdot$] $ Q_i \cap \partial \Omega=\{(y_1,\beta_i(y_1): y_1 \in (-\delta_\Omega,\delta_\Omega) \}$, \item[$\cdot$] $Q_i \cap  \Omega=\{(y_1,y_2) \in Q_i: y_2>\beta_i(y_1) \}$,
\end{itemize}
\item[$\cdot$] $\Omega$ has the strong local Lipschitz property with constants $M_\Omega, \delta_\Omega, N_\Omega $.
\end{itemize}
\end{proposition}
\begin{proof} It is known (see \cite[Corollary 1.2.2.3]{Grisvard} that  $\Omega$ bounded and convex implies that $\Omega$ has a Lipschitz boundary (in the sense of  \cite[Definition 1.2.2.1]{Grisvard}), which is exactly what is described in the first two assertions. The fact that $\beta_i$ is convex is a consequence of the fact that its epigraph $Q_i \cap  \Omega$ is convex. Finally, the first two assertions imply the strong local Lipschitz property with constants $M_\Omega, \delta_\Omega, N_\Omega $ as described in \cite[IV.4.2]{Adams}.
\end{proof}

\subsubsection{Approximation by smooth convex subdomains} \label{approxdom} We    construct   a  sequence of smooth convex domains $\Omega_n$
increasing to $\Omega$, that is
\begin{equation} \label{gammacv}
\Omega_n  \textrm{ smooth convex}, \qquad \Omega_n \Subset \Omega_{n+1} \Subset \cdots \Subset \Omega, \qquad \Omega= \bigcup_{n} \Omega_n.
\end{equation}
and with the property that
\begin{equation} \label{gammacv2}
\textrm{the constants in the strong local Lipschitz property of } \;\Omega\; \textrm{ are uniform in } n.
\end{equation}
We introduce the Minkowski functional of $\Omega$ $$\mu_\Omega: \overline{\Omega} \to [0,\infty), \qquad   \mu_\Omega(x)= \inf\;\{\lambda>0: \lambda^{-1}x \in \overline{\Omega}\} .$$ The function $\mu_\Omega$ is   a convex  function on $\overline{\Omega}$ (see \cite[pp.\ 57-59]{Hormander}). A convex function on a compact subset of any normed space is globally Lipschitz (see \cite{RV} for a  simple proof): thus, call $L_\Omega$ the Lipschitz constant of $\mu_\Omega$. The function $\rho=\mu_\Omega-1$ is convex and Lipschitz with the same constant $L_\Omega$, and
$$
 \Omega = \{x \in \overline{\Omega}: \rho(x)<0\},\qquad\partial \Omega = \{x \in \overline{\Omega}: \rho(x)=0\}.
$$
For $\eps>0,$ let $\Omega^\eps $ be the $\eps$-neighborhood of $\Omega$. It is easy to verify that the mollification (see Section \ref{nota} for notation)   $$\rho_\eps: \Omega^\eps \to [-1,0], \qquad \rho_\eps:= \rho *\phi^2_{\eps}  $$ is  smooth and convex, and moreover that the Lipschitz constants of $\{\rho_\eps: \eps>0\}$ are uniformly bounded by $L_\Omega$. Finally, we have that, for $\varepsilon\rightarrow 0, \rho_\eps \to \rho$ in the uniform Lipschitz norm, that is $$
 \sup_{x,y \in \Omega} \frac{|\rho_\eps(x) -\rho(x)-\rho_\eps(y)+\rho(y)|}{|x-y|} \to 0, \qquad \eps \to 0.
$$
Choose a subsequence $n_k$ with $$
 \sup_{x \in \Omega}\big|\rho_{\frac{1}{n_k}}(x)- \rho_{\frac{1}{n_{k+1}}}(x)\big|< (n_k)^{-3},$$
and define
$$
\Omega_{k}:= \Big\{ x \in \Omega: \rho_{\frac{1}{n_k}}(x)<-\textstyle\frac{1}{n_k}\Big\};
$$
 it follows that $\Omega_k$ is a convex open set, $\Omega_{k} \Subset \Omega_{k+1} \Subset \Omega$, and that  $\bigcup\Omega_{k}=\Omega$. Moreover the $\partial \Omega_k$ are smooth, and uniformly Lipschitz (with respect to $k$). Thus each $\Omega_k$ has the strong local Lipschitz property with Lipschitz constant uniformly bounded in $k$. The construction \eqref{gammacv2} is thus completed.

 Several important consequences of \eqref{gammacv}-\eqref{gammacv2} will be derived in the next subsections. Here, we mention that \eqref{gammacv2} guarantees that the implicit constants appearing in the Sobolev embeddings and trace theorems on  $\Omega_n$, which depend on the constants in the strong local Lipschitz property (see Theorem IV.4.1 and its proof in \cite{Adams} for instance) are uniform in $n$ (they do depend on $\Omega$ however,  through $M_\Omega, \delta_\Omega, N_\Omega$ and $L_\Omega$).

\subsection{Normal vector, normal traces, tangential vector fields}   Maintaining the notation of Proposition \ref{cvx1}, we observe that for each $i$,
$
\beta_i'(y_1)
$ is defined almost everywhere on $(-\delta_\Omega,\delta_\Omega)$ and $|\beta_i'(y_1)|\leq M_\delta$ wherever defined. We can thus define the normal vector almost everywhere on $Q_i\cap \partial \Omega$ by
$$
\cic{n}(x)=\cic{n}(y_1,\beta_i(y_1)) := \frac{(-\beta'_i(y_1),1) }{\sqrt{1+( \beta'_i(y_1))^2}}, \qquad x= (y_1,\beta_i(y_1)) \in  Q_i\cap \partial \Omega.
$$
This definition can be extended to all of $Q_i$ by $\cic{n}(y_1,y_2)= \cic{n}(y_1,\beta_i(y_1))$, and using a partition of unity subordinated to the covering of $\overline\Omega$ by the $Q_i,$  to a bounded vector field on all of $\overline{\Omega}$.

Thus, for $\v \in \D(\overline\Omega; \R^2), $
$$
\gamma_{\mathbf{n}} \v (x) := \v(x) \cdot \mathbf{n} (x), \quad \gamma_{\mathbf{t}} \v (x) := \v^\perp(x) \cdot \mathbf{n} (x) , \qquad x \in \partial \Omega
$$
is defined almost everywhere on $\partial\Omega$.  We are interested in the spaces
$$
\cic{L}_{\div}^p(\Omega) = \{\v\in \cic{L}^p(\Omega) : \div \v \in L^p(\Omega)\}, \qquad 1\leq q\leq \infty.
$$
The classical $W^{1,p}(\Omega)\to W^{1-\frac1q,q}(\partial\Omega) $ trace theorem due to Gagliardo \cite{Gagl} yields the lemma below, arguing along the same lines of \cite[Theorem I.1.2]{TEM2}. For a definition of Sobolev spaces (of fractional order) on Lipschitz submanifolds of $\R^d$, see Subsection 1.3.3 in \cite{Grisvard}. Note that if   $f \in \D(\overline\Omega)$ (that is, $f$ is restriction to $\overline\Omega$ of a function in $\C^\infty(\R^2)$), then the restriction of $f$ to $\partial \Omega$ is a Lipschitz function on the Lipschitz submanifold $\partial \Omega$.
\begin{lemma}  \label{TF-L2} Let $\Omega\subset \R^2$ be a bounded convex domain and $p \in (1,\infty)$.
The map  $  \gamma_{\mathbf{n}}: \v \in \D(\overline\Omega; \R^2) \to \mathbf{\gamma}_{\mathbf{n}}\mathbf{v} \in\mathrm{Lip}(\partial \Omega)  $ extends  as  a linear   bounded map
$$
  \gamma_{\mathbf{n}} : \cic{L}_{\div}^p(\Omega)   \to W^{-\frac{1}{p'},p'} (\partial \Omega)
$$
and the generalized Stokes formula
\begin{equation} \label{TF-stokes}
 (\v, \nabla\varphi)_{\Omega} + (\div  \v, \varphi)_{\Omega} = \l \gamma_{\mathbf{n}} \v , \gamma_0\varphi\r
\end{equation}
holds for every $\v \in  \cic{L}_{\div}^p(\Omega)$ and $\varphi \in W^{1,p'}(\Omega)$.
\end{lemma}
\begin{remark} If $\Omega_n$ is an approximating sequence of domains as in Subsection \ref{approxdom}, the norms of  the $\gamma_{\mathbf{n}}      \in \mathcal{L}( \cic{L}_{\div}^p(\Omega)    \to W^{-1/p',p'} (\partial \Omega_n))$ are uniformly bounded in $n$. This uniformity descends from the uniformity of the constants in the Gagliardo trace theorem, \cite{Gagl}.\end{remark}
\subsubsection{$\cic{L}^p$-tangential vector fields} We say that $\v\in \cic{L}^1(\Omega)$ is a tangential divergence free vector field if
\begin{equation} \label{TVF}
 (\v, \nabla\varphi)_{\Omega} =0 \qquad \forall \varphi \in \D(\overline{\Omega}).
\end{equation}
As a consequence of Lemma \ref{TF-L2}, it follows that
\begin{equation} \label{TVFlemma}
\cic{L}_{\tau}^p(\Omega):=\{\v \in \cic{L}_\div^p(\Omega): \eqref{TVF} \textrm{ holds}\}= \{\v \in \cic{L}_{\div}^p(\Omega): \div \v=0, \gamma_\cic{n} \v =0\}.
\end{equation}

\subsection{Helmholtz decomposition} Let $\V:=\{ \v \in \D(\Omega; \R^2) : \div \v=0  \}$. It is well known (see for example Theorem I.1.4 in \cite{TEM2})  that 
for $\Omega$ Lipschitz
\begin{equation} \label{helm-l2}
\cic{L}_{\tau}^2(\Omega)= \textrm{ the closure of } \V  \textrm{ in } \cic{L}^2(\Omega), \qquad \cic{L}_{\tau}^2(\Omega)^\perp = \{\nabla \pi : \pi \in H^1(\Omega)\}.
\end{equation}
Let us denote by $P_\Omega: \cic{L}^2(\Omega) \to \cic{L}_{\tau}^2(\Omega)$ the corresponding orthogonal projector. The following result, 
due to Geng and Shen \cite{Geng-Shen} allows us to obtain  \eqref{helm-l2} for all $1<p<\infty$ and extend $P_\Omega$ boundedly to $\cic{L}^p(\Omega)$. 
Note that Theorem I.1.4 in \cite{TEM2} (i.e.\ the case $p=2$ of   Theorem \ref{helmthm} below) holds whenever $\Omega$ is Lipschitz; actually, the range $p\in \big(\frac32-\eps,3+\eps\big)$ is known to be sharp for general Lipschitz domains   \cite{FMM}.
\begin{theorem} Let $\Omega \subset \R^2$ be a bounded convex domain and $1<p<\infty$. The following hold:  \label{helmthm}
\begin{equation} \label{helm-lq}
\cic{L}_{\tau}^p(\Omega)= \textrm{ \emph{the closure of} } \V  \textrm{ in } \cic{L}^p(\Omega);  \end{equation}
$P_\Omega$ extends to a bounded linear operator $$P_\Omega :\cic{L} ^p(\Omega) \to \cic{L}_{\tau}^p(\Omega), 
\qquad   P_\Omega|_{\cic{L}_{\tau}^p(\Omega)} =\mathrm{I}_{\cic{L}_{\tau}^p(\Omega)}, $$  with operator norm only depending on 
$M_\Omega,$  $\delta_\Omega,$ $N_\Omega,$ $L_\Omega$. Moreover,
for each $\v \in \cic{L}^p(\Omega)$, there exists $ \pi \in W^{1,p}(\Omega)$, unique up to an additive constant such that
\begin{equation}
P_\Omega^\perp \v:=(I-P_\Omega) \v= \nabla \pi.
\label{iminusp}
\end{equation}
\end{theorem}
\begin{proof} The second and third statements appear almost verbatim in \cite[Theorem 1.3]{Geng-Shen}. 
Let us show how they imply \eqref{helm-lq}; denote by $\cic{H}_\tau^p$ the space on the right hand side of \eqref{helm-lq}. The backward inclusion is easy 
(see the proof of \cite[Theorem I.1.4]{TEM2} for example). To get equality, we begin by characterizing the annihilator of 
$ \cic{H}_\tau^p$ in $\cic{L}_\tau^p(\Omega)$
as
\begin{equation}
\big(\cic{H}_\tau^p\big)^\perp:= \big\{F \in \mathcal{L}(\cic{L}_\tau^p(\Omega), \R): F|_{\cic{H}_\tau^p}=0\big\} = \nabla W^{1,p'}(\Omega).
\label{anni}
 \end{equation}
\begin{proof}[Proof of \eqref{anni}] It is known (see \cite{Xwang} for a simple proof) that $$
F\in \D'(\Omega;\R^2), \qquad F(\v)= 0 \quad \forall \v \in \V \implies F= \nabla \pi, \; \pi \in\D'(\Omega).$$
Now from the Riesz representation theorem, for each $F   \in \mathcal{L}(\cic{L}_\tau^p(\Omega), \R) $ there exists a (possibly nonunique) ${\f} \in \cic{L}^{p'}(\Omega)$ such that $F(\v)=({\f},\v)_\Omega$ for all $\v \in  \cic{L}_\tau^p(\Omega)$. From the above characterization, it follows that ${\f}= \nabla \pi.$ {The Deny-Lions characterization (see for example \cite[Proposition I.1.1]{TEM2}) gives that
$ \pi \in L^2(\Omega)$; since $\nabla \pi \in \cic{L}^{p'}(\Omega)$, the   Poincar\'e inequality finally yields $  \pi \in {W}^{1,p'}(\Omega)$}.
\end{proof} \noindent
With \eqref{anni} in hand, we show that $\big(\cic{H}_\tau^p\big)^\perp=\{\cic{0}\}$, thus proving \eqref{helm-lq}. Let $F \in \big(\cic{H}_\tau^p\big)^\perp$, and ${\f} \in \cic{L}^{p'}(\Omega) $ as above. By \eqref{iminusp}, it follows that ${\f}= (I-P_\Omega) {\f}$, i.e.\ $P_\Omega {\f}=0$. Therefore
$$
({\f},\v)_\Omega =  ({\f},P_
\Omega \v)_\Omega = (P_
\Omega {\f}, \v)_\Omega =0 \qquad \forall\, \v \in  \cic{L}_\tau^p(\Omega),
$$
and this proves the last claim. Here we used the same notation (with slight abuse) for both $
 P_\Omega : \cic{L}^p(\Omega) \to \cic{L}^p(\Omega)$ and $ P_\Omega : \cic{L}^{p'}(\Omega) \to \cic{L}^{p'}(\Omega) .$
\end{proof}
\subsection{The Dirichlet problem and the Biot-Savart law}
Denote by
$$
A_\Omega= -\Delta \textrm{ on } \Omega \textrm{ with Dirichlet boundary conditions}, \qquad \mathrm{dom}(A_\Omega)= H^2(\Omega)\cap H^1_0(\Omega).$$
We recall the following consequence of the Lax-Milgram lemma.
\begin{proposition} \label{ellprob} Let  $f \in H^{-1}(\Omega)$ be given. Then there exists a unique $F \in H^1_0(\Omega)$ satisfying
\begin{equation} \label{laxmilgram}
A_\Omega F = f, \qquad \|F\|_{H_0^1(\Omega)} \leq \|f\|_{H^{-1}(\Omega)}.
\end{equation}
\end{proposition}
Referring to \eqref{laxmilgram}, we   use the notation $ F={\mathrm{G}}_\Omega f$. The classical theory of elliptic equations (see for example \cite{Gilbarg}) 
tells us that, when $\Omega$ is a bounded smooth domain, and  $f \in L^p(\Omega)$,  ${\mathrm{G}}_\Omega f$ has two derivatives in $L^p(\Omega)$ 
for any $1<p<\infty$. This is no longer true in general if the domain $\Omega$ is merely bounded and convex. However, the above result still holds in the 
range $1<p\leq2$: we state this precisely in the theorem below, due to Fromm \cite{Fromm}.
\begin{theorem}
\label{ercvx} Let $\Omega\subset \R^2$ be a bounded convex domain.
Let  $ 1<p\leq 2$ and  $f\in L^p(\Omega)$ be given.  Then, there exists a positive constant $C_{p,\Omega}$, depending only on $p$ and on the Lipschitz 
character of  $\Omega$, such that
\begin{equation}
\label{epest1}
\mathrm{G}_\Omega f \in W^{2,p}(\Omega) \cap W^{1,p}_0(\Omega), \qquad
 \|{{\mathrm{G}}}_{\Omega}f\|_{W^{2,p}(\Omega)} \leq C_{p,\Omega} \|f\|_{L^p(\Omega)}.
\end{equation}
\end{theorem}
Let $\Omega_n$ be the  approximating sequence to $\Omega$ constructed in \eqref{gammacv}. In this context, for a function $F :\Omega_n \to \R$, 
we denote by $\widetilde F$ its extension by zero to $\Omega$. Note that 
\begin{equation} \label{extensionW}
  F \in W^{1,p}_0(\Omega_n) \implies \widetilde F \in W^{1,p}_0(\Omega )
\quad \textrm{ and } \quad\|\widetilde F\|_{W^{1,p}_0(\Omega )} = \| F\|_{W^{1,p}_0(\Omega_n )}.
\end{equation}  
We state and prove a so-called 
$\Gamma$-convergence type result. The restrictions $f_n,f \in L^1_{\mathrm{loc}}(\Omega)$ can be lifted, but we only need the particular case contained 
in the lemma below. 
\begin{lemma}\label{lemmaconvcvx} Let $f_n,f \in L^1_{\mathrm{loc}}(\Omega)$ and set $(f_n)_{\Omega_n} =f_n \cic{1}_{\Omega_n}. $ Then
\begin{equation} \label{convH-1} f_n \to f\quad  \textrm{\emph{in}} \; H^{-1}(\Omega) \implies
\widetilde{\mathrm{G}_{\Omega_n} (f_n)_{\Omega_n}} \to \mathrm{G}_{\Omega}  f\quad  \textrm{\emph{in}} \;    H^1_0(\Omega).
\end{equation}
\end{lemma}
\begin{proof}   We first show that $\widetilde{\mathrm{G}_{\Omega_n} (f_n)_{\Omega_n}} \to \mathrm{G}_{\Omega}  f$ weakly in $H^1_0(\Omega)$.  By density of $ \varphi \in\D(\Omega)$  in $H^1_0(\Omega) $  it suffices to show that
$$
\lim_{n \to \infty}\big(\nabla\widetilde{\mathrm{G}_{\Omega_n} (f_n)_{\Omega_n}} -\nabla \mathrm{G}_{\Omega}  f , \nabla\varphi\big)_\Omega  =0
 \qquad \forall \varphi \in\D(\Omega).$$
Fix $ \varphi \in\D(\Omega)$ and choose $N$  such that $\supp \varphi:= D \Subset \Omega_N$. Then for $n \geq N, \supp \varphi\subset\Omega_n$ and 
\begin{align*}
\big(\nabla\widetilde{\mathrm{G}_{\Omega_n} (f_n)_{\Omega_n}}  , \nabla\varphi\big)_\Omega 
& = \big(\nabla\widetilde{\mathrm{G}_{\Omega_n} (f_n)_{\Omega_n}}  , \nabla\varphi\big)_{\Omega_n} = 
\big(\nabla{\mathrm{G}_{\Omega_n} (f_n)_{\Omega_n}}  , \nabla\varphi\big)_{\Omega_n}\\ 
&= -   (\Delta \mathrm{G}_{\Omega_n} (f_n)_{\Omega_n}, \varphi )_{\Omega_n}\\ 
& = -(f_n, \varphi)_{\Omega_n}= -(f_n, \varphi)_\Omega, 
\end{align*}
and similarly $(\nabla \mathrm{G}_{\Omega}  f , \nabla\varphi\big)_\Omega=-(f , \varphi)_\Omega $. Thus
$$\big(\nabla\widetilde{\mathrm{G}_{\Omega_n} (f_n)_{\Omega_n}} -\nabla 
\mathrm{G}_{\Omega}  f , \nabla\varphi\big)_\Omega  =  ( f_n-f, \varphi )_\Omega = \l f_n-f, \varphi \r \to 0, $$
and the weak convergence follows. Moreover, we also see that \begin{align*}
\|\widetilde{\mathrm{G}_{\Omega_n} (f_n)_{\Omega_n}} \|_{H^1_0(\Omega)}^2  & =  \big( \widetilde{\mathrm{G}_{\Omega_n} (f_n)_{\Omega_n}}, f_n\big)_\Omega = \big( \widetilde{\mathrm{G}_{\Omega_n} (f_n)_{\Omega_n}}, f_n-f \big)_\Omega +  \big( \widetilde{\mathrm{G}_{\Omega_n} (f_n)_{\Omega_n}},  f \big)_\Omega \\ &   \to 0+  \big( \mathrm{G}_{\Omega } f, f \big)_\Omega = \| \mathrm{G}_{\Omega } f\|_{H^1_0(\Omega)}^2   
\end{align*}  
which allows the upgrade from weak to strong convergence in  $H_0^1(\Omega) $ of $\widetilde{\mathrm{G}_{\Omega_n} (f_n)_{\Omega_n}}$ to $ \mathrm{G}_{\Omega}  f$, thus finishing the proof of the lemma.
\end{proof}
We now make   the connection between the Euler system and $A_\Omega$ more explicit. Set
 \begin{equation}
\mathbf{K}_\Omega f := \nabla^\perp ( \mathrm{G}_\Omega f), \qquad f \in H^{-1}(\Omega).
\end{equation} We have that
 \begin{align}
& \mathbf{K}_\Omega\in \mathcal{L}(H^{-1}(\Omega), \mathbf{L}^2_\tau(\Omega)),  \label{K-boundweak}\\
& \mathbf{K}_\Omega \in \mathcal{L}\big(L^p(\Omega), \mathbf{W}^{1,p} (\Omega))\cap \mathbf{L}^2_\tau(\Omega)\big), \label{K-boundstrong} \qquad 1<p\leq 2.
\end{align}
\begin{proof}[Proof of \eqref{K-boundweak}-\eqref{K-boundstrong}]
Due  to \eqref{epest1} and \eqref{laxmilgram}, we are only left to verify that
\begin{equation} \label{tangential-K}
\div \mathbf{K}_\Omega f =0, \qquad \gamma_{\mathbf{n}} \mathbf{K}_\Omega f = 0, \qquad \forall f \in H^{-1}(\Omega).
\end{equation}
 Let $\varphi \in \D(\overline{\Omega})$. We then have, integrating by parts,
$$
(\mathbf{K}_\Omega f, \nabla \varphi)_\Omega = (\nabla ( \mathrm{G}_\Omega f), \nabla^\perp \varphi)_\Omega 
= -( \mathrm{G}_\Omega f, \div \nabla^\perp \varphi)_\Omega + \int_{\partial\Omega} (\mathrm{G}_\Omega f )\nabla^\perp \varphi \cdot \mathbf{n}, 
$$
and both terms vanish in the right-hand side. We conclude by means of \eqref{helm-lq}.
\end{proof}

The next lemma shows that $\curl$   is the left inverse of $ \nabla^{\perp}\circ \mathrm{G}_\Omega$ on $L^p(\Omega)$, $1<p\leq 2$.
\begin{lemma}\label{nablaperplemma}
Let  $ 1<p\leq 2$ and  $f \in L^p(\Omega)$. Then
\begin{equation} \label{nablaperp}
 f  =  \curl \mathbf{K}_\Omega f.
\end{equation}
\end{lemma}
\begin{proof} By density, it  suffices to show that
\begin{equation}
\label{nablaperpdense}
(f, \varphi)_\Omega = \left( \curl \mathbf{K}_\Omega f, \varphi \right),  \qquad \forall \varphi  \in  \D(\Omega).
\end{equation}
  Let now $\mathbf{\varphi} \in \D(\Omega) $.
 We have
\begin{align*} (f, \varphi )_\Omega&=  (-\Delta \mathrm{G}_\Omega f , \varphi)_\Omega = (\mathrm{G}_\Omega f , -\Delta \varphi)_\Omega      
=  (\nabla^\perp\mathrm{G}_\Omega f,  \nabla^\perp \varphi)_{\Omega} \\ &= (\mathbf{K}_\Omega f,  \nabla^\perp \varphi)_{\Omega} 
= (\curl \mathbf{K}_\Omega f,    \varphi)_{\Omega} - \int_{\partial \Omega} \varphi(\mathbf{K}_\Omega f)^\perp \cdot \mathbf{n}, 
\end{align*}
and the last term on the right hand side is zero.
We integrated by parts in the last equality, which is legitimate since $\mathbf{K}_\Omega f \in \mathbf{W}^{1,p}(\Omega)$.
\end{proof}
 
\subsection{The spaces $\cic{V}^{1,p}$} The evolution of our solution to the Euler system \eqref{P} will take place in the space of $\cic{L}^p$ 
tangential vector-fields with $L^p$ vorticity. That is, we define, for $1<p<\infty$, 
$$ 
\cic{V}^{1,p}(\Omega):= \big\{\v \in \cic{L}^p_\tau (\Omega) : \curl \v \in L^p(\Omega) \big\}, 
\; \|\v\|_{\cic{V}^{1,p}(\Omega)}  := \big( \|\v\|_{\cic{L}^{ p}(\Omega)}^p + \|\curl \v\|_{L^p(\Omega)}^p \big)^{\frac1p}.
$$
As a consequence of Lemma \ref{nablaperplemma} and Theorem \ref{ercvx}, when $1<p\leq 2$, we have the continuous embedding 
$\cic{V}^{1,p}(\Omega) \hookrightarrow \cic{W}^{1,p}(\Omega)  $:
\begin{equation}
\label{embedding-1q2}
\|\v\|_{\cic{W}^{1,p}(\Omega)} \leq C_{p,\Omega} \|\v\|_{\cic{V}^{1,p}(\Omega)} \qquad \forall \v \in \cic{V}^{1,p}(\Omega), \; 1<p\leq 2.
\end{equation}
 The embedding  discussed above allows for an improvement of the boundary regularity of functions in $\cic{V}^{1,p}(\Omega)$, by further applying 
 Gagliardo's trace theorem to the components of $\v$. More precisely,
 \begin{equation} \label{trace1}
 \v|_{\partial \Omega} \in \cic{W}^{1-\frac1p,p'}(\partial \Omega), \qquad \forall 1<p\leq 2.
\end{equation}
In view of \eqref{trace1}, we can therefore make sense of $\gamma_\cic{n} \v$ as $\v \cdot \cic{n}|_{\partial \Omega}$ whenever $\v \in \cic{V}^{1,p}(\Omega)$.

\section{Elliptic regularity in a  rectangle} \label{Sec3} We assume throughout this section that $\Omega=[0,L_1]\times [0,L_2]$. Hereafter, we develop an appropriate substitute of Theorem  \ref{ercvx} in the range $2<p\leq\infty$, with spaces of functions with bounded mean oscillation replacing $L^\infty$.
\subsection{Local bmo spaces}
We denote by ${\mathrm{bmo}(\R^2)}$ the strict subspace of $\mathrm{BMO}(\R^2)$ normed by
\begin{equation}
\label{bmonorm}
\|f\|_{\mathrm{bmo}(\R^2)} = \sup_{|Q| < 1} \frac{1}{|Q|}\int_{Q} |f(x)-f_Q|\, \d x +\sup_{|Q| \geq 1} \frac{1}{|Q|}\int_{Q} |f(x)|\, \d x  
\end{equation}
where the suprema above are taken over squares $Q \subset \R^2 $, $|Q|$ denotes the area of $Q$,  and $f_Q$ denotes the average of $f$ on the cube $Q$.
See for example \cite{CDS} for more details.  

Let $D \subset \R^2$ be  a domain.
For a function $f:  \overline  D \to \R,$ let $\overline f$ be its extension to zero outside $\overline D $, i.e. $\overline  f:=f\cic{1}_{\overline D}$. We define the Banach spaces
$$
\mathrm{bmo}_z( {D}) = \{f: \overline {D} \to \R: \overline f \in \mathrm{bmo}(\R^2) \}, \qquad  \|f\|_{\mathrm{bmo}_z({D})} := \|\overline f\|_{\mathrm{bmo}(\R^2)},
$$i.e.\
the space of functions on $\overline{D}$ whose trivial extension is in ${\mathrm{bmo}(\R^2)}$, and 
\begin{align*}
&\mathrm{bmo}_r( {D}) = \{f: {D} \to \R: \exists F \in \mathrm{bmo}(\R^2) \textrm{ with } F|_{D}=f \}, \\ & \|f\|_{\mathrm{bmo}_r({D})} :=  \inf_{\substack{F \in \mathrm{bmo}(\R^2) \\ F|_D=f}}\|F\|_{\mathrm{bmo}(\R^2)},
\end{align*}
i.e. the space of functions on ${D}$ which are restrictions to $D$ of functions in ${\mathrm{bmo}(\R^2)}$.
The continuous embeddings $
L^\infty({D})  \hookrightarrow\mathrm{bmo}_z( {D})\hookrightarrow\mathrm{bmo}_r( {D}) $ are immediate to verify; this, together with  John-Nirenberg's inequality
\begin{equation} \label{JN}
\|f\|_{L^p({D})} \leq C_{D} p \|f\|_{\mathrm{bmo}_r( {D}) },
\end{equation}
where the constant $C_{D}$ is only dependent on $\mathrm{diam}({D})$, hints at the relevance of  $\mathrm{bmo}_z( {D})$ and $\mathrm{bmo}_r( {D})$   as a substitute  for $L^\infty({D})$. For $\star\in\{z,r\}$, we use the notation
\begin{align*}  &
W^{2,\mathrm{bmo}_\star}(D)= \{f \in \mathrm{bmo}_\star(D): D^2 f \in  \mathrm{bmo}_\star(D)^{2\times 2}\}, \\ & \|f\|_{W^{2,\mathrm{bmo}_\star}(D)}:= \|  f\|_{\mathrm{bmo}_\star(D)  } 
+\|D^2 f\|_{\mathrm{bmo}_\star( D)^{2\times2} }.
\end{align*}
The next theorem, which we quote from \cite{CDS},  tells us that the solution to the Dirichlet problem $\mathrm{G}_D f$ is in
$
W^{2,\mathrm{bmo}_z}(D)$
whenever $f \in \mathrm{bmo}_z(D)$.
\begin{theorem}
\label{stein-thm} Let $D$ be  either a bounded domain of class $\C^2$ or the halfspace $\R^2_+$. Then there exists a constant $C_D>0$, depending only on $D$, such that\begin{equation}
\label{epestbmohp}
\mathrm{G}_{D}: {\mathrm{bmo}_z(D) }  \to   W^{2,\mathrm{bmo}_z}(D),   \qquad
 \|{{\mathrm{G}}}_{D}f\|_{W^{2,\mathrm{bmo}_z}(D)} \leq C_D  \|f\|_{\mathrm{bmo}_z( D) }.
\end{equation}
\end{theorem}
\subsection{Theorem \ref{stein-thm} in a rectangle} We prove the following proposition, which is a (partial) extension of Theorem \ref{stein-thm} to the (nonsmooth, convex) domain $\Omega=[0,L_1] \times [0,L_2]$.
\begin{proposition} Let $\,\Omega=[0,L_1] \times [0,L_2]$. There exists $C=C(L_1,L_2)$ such that \label{bmolaw}
\begin{equation} 
\label{epest3}  \|{{\mathrm{G}}}_{\Omega}f\|_{W^{2,\mathrm{bmo}_r}(\Omega)} \leq C \|  f\|_{\mathrm{bmo}_z( \Omega)  }. 
\end{equation}  
\end{proposition}
\begin{proof} 
In this proof, the almost inequality sign is hiding a positive constant, possibly varying from line to line and  depending on $\Omega$ only. We also refer to Remark \ref{remarkpoly} for notation.

Denote by $F:= \mathrm{G}_{\Omega}f$.
We preliminarily observe that, by Grisvard's classical $L^p-W^{2,p}$ estimate for domains with corners of aperture less than or equal to $\frac\pi2$ (which we recall in Section \ref{s6}, see  \eqref{epestpoly}) applied for $p=3$, and John-Nirenberg's inequality \eqref{JN}
$$
\|F\|_{W^{2,3}(\Omega)}\lesssim  \|f\|_{L^3(\Omega)} \lesssim  \| f\|_{\mathrm{bmo}_z(\Omega)}.
$$
By the Sobolev embedding $W^{2,3}(\Omega) \subset \C^{1,\frac13}(\overline \Omega)$,  we have in particular that
\begin{equation}
\label{bootstrap} \nabla F \in \C^{0,\frac13}(\overline\Omega)^2, \qquad 
\|F\|_{L^\infty(\Omega)} + \|\nabla F\|_{\cic{L}^\infty(\Omega)} \lesssim \| f\|_{\mathrm{bmo}_z(\Omega)}.
\end{equation}

The core of the argument begins now. Let $ \C:=\big\{\Omega_0 \cup \{B_i: i =1, \ldots, 4\}\big\}$  be an  open cover of $\overline \Omega$ as follows:  each $ B_i $ is an open ball centered at the corner $S_i$,  of  radius $\rho>0$ chosen small enough so that $B_i \cap B_j =\emptyset$ when $1 \leq i <j \leq 4$; $\Omega_0 \subset \Omega$ is an open set with smooth boundary and $
 \mathrm{dist}(\overline{\Omega_0}, \{S_1,\ldots, S_4\}) > \frac \rho 2.
$
For $i=1,\ldots,4$, set   $\Omega_i=B_i \cap \Omega$.
Let $\{\mu_0, \mu_1, \ldots, \mu_4 \}$ be a partition of unity subordinated to the cover $\C$, and write $F_i= F\mu_i$. Then $F _i$ solves the Dirichlet problem
\begin{equation}
\label{ep-ith}  
\begin{cases}
\Delta F_i = f_i & \textrm{on } \Omega_i, \\
F_i=0 &  \textrm{on } \partial  \Omega_i , 
\end{cases}\qquad 
f_i= \mu_i f+ \nabla F \cdot \nabla \mu_i + F \Delta \mu_i, \qquad i=0,\ldots, 4.
\end{equation} 
We gather from \eqref{bootstrap} that  
\begin{equation}
\label{bootstrap2} \|f_i\|_{\mathrm{bmo}_z(\Omega)} \lesssim  \|\mu_i\|_{C^2(\Omega)} \big( \|f\|_{\mathrm{bmo}_z(\Omega)} + \|\nabla F\|_{\cic{L}^\infty(\Omega)} + \| F\|_{L^\infty(\Omega)} \big) \lesssim \|f\|_{\mathrm{bmo}_z(\Omega)}.
\end{equation}
Since $  \Omega_0$ is a smooth domain, we apply Theorem \ref{stein-thm} with $D=\Omega_0$, and get
\begin{equation}
\label{zeroth} \|F_0\|_{W^{2,\mathrm{bmo}_r}(\Omega)} =  \|F_0\|_{W^{2,\mathrm{bmo}_z}(\Omega_0)}\lesssim \|f_0\|_{\mathrm{bmo}_z(\Omega_0)} \lesssim  \|f\|_{\mathrm{bmo}_z(\Omega)}.
\end{equation}
We now deal with the cases $i\geq 1$ and  show that
\begin{equation}
\label{ith} \|F_i\|_{W^{2,\mathrm{bmo}_r}( \Omega)} \lesssim \|f_i\|_{\mathrm{bmo}_z(\Omega)}, \qquad i=1,\ldots, 4
\end{equation}
which will be enough to obtain  \eqref{epest3}. 
The four corners are treated in the exact same way, so we fix  $i=1$, and $S_1=(0,0)$.  Note that $F_1$ solves the Dirichlet problem 
$$
\begin{cases}
\Delta F_1 = f_1 & \textrm{on } \Gamma: (0,\infty) \times (0,\infty), \\
F_1(x_1,0)=0 &   x_1 \geq 0, \\
F_1(0,x_2)=0 &   x_2 \geq 0. 
\end{cases} 
$$  
Consider the odd reflections of $F_1$ and $f_1$, defined on $\R^2_+$ by:$$
w(x_1,x_2):= \mathrm{sign}(x_1) F_1(|x_1|,x_2) , \qquad h(x_1,x_2):=  \mathrm{sign}(x_1) f_1(|x_1|,x_2).
$$
One sees that $w$ is the solution to the Dirichlet problem on $\R^2_+$  with data $h$ and  zero boundary condition, i.e.\ $w= \mathrm{G}_{\R^2_+} h $. We note  that $\|h\|_{\mathrm{bmo}_z(\R^2_+)} \leq 2 \|f_1\|_{\mathrm{bmo}_z(\Gamma)};$ indeed, this is    the same as  $\|\overline h\|_{\mathrm{bmo}(\R^2)} \leq 2 \|\overline{f_1}\|_{\mathrm{bmo}(\R^2)}$; recall that the bar denotes extension by zero to $\R^2$. This is clear, since $\overline h (x_1,x_2) = \overline{f_1} (x_1,x_2) - \overline{f_1}(-x_1,x_2). $ Therefore, we can apply Theorem \ref{stein-thm} in the case of $D=\R^2_+$, and deduce that
\begin{equation}
\label{ith-sq3} \|w\|_{W^{2,\mathrm{bmo}_z}( \R^2_+)} \lesssim  \|h\|_{\mathrm{bmo}_z(\R^2_+)}\lesssim  \|f_1\|_{\mathrm{bmo}_z(\Gamma)}= \|f_1\|_{\mathrm{bmo}_z(\Omega)}.
\end{equation}
Since  $w$ restricted to $\Gamma$ coincides with $F_1$, and $F_1$ is supported in $\Gamma$, we have $$ \|F_1\|_{W^{2,\mathrm{bmo}_r}( \Omega)} =  \|F_1\|_{W^{2,\mathrm{bmo}_r}( \Gamma)}  \leq  \|w\|_{W^{2,\mathrm{bmo}_z}( \R^2_+)}, $$ and thus \eqref{ith} follows from  \eqref{ith-sq3}. The proof of Proposition \ref{bmolaw} is complete.
\end{proof}

\subsection{The space $\cic{V}^{1,\infty}(\Omega)$} In view of the elliptic regularity result \eqref{epest3}, for $\Omega=[0,L_1]\times [0,L_2]$ it makes sense to extend the scale of spaces $\cic{V}^{1,p}(\Omega)$ to the endpoint $p=\infty$ by setting
$$
\cic{V}^{1,\infty}(\Omega):= \big\{\v \in \cic{L}^2_\tau (\Omega) : \curl \v \in L^\infty(\Omega) \big\}, 
\quad \|\v\|_{\cic{V}^{1,\infty}(\Omega)}  =  \|\v\|_{\cic{L}^{ \infty}(\Omega)}  + \|\curl \v\|_{L^\infty(\Omega)}.
$$
As a consequence of Proposition \ref{bmolaw}, we have the continuous embedding
 \begin{equation}
\label{Wtauinftyemb}
\cic{V}^{1,\infty}(\Omega) \hookrightarrow \cic{W}^{1,\mathrm{bmo}_z}(\Omega)=\{\v \in \mathrm{bmo}_z(\Omega)^2: \nabla \v \in  \mathrm{bmo}_z(\Omega)^{2\times 2}\}
\end{equation}
and in turn, John-Nirenberg inequality gives in particular that \begin{equation}
\label{JNineq}
\|\nabla \v\|_{L^p(\Omega)^{2 \times 2}} \leq Cp \|\v\|_{\cic{V}^{1,\infty}(\Omega) } \qquad \forall p \in [1,\infty).
\end{equation}
A further consequence of Proposition \ref{bmolaw} is that,   when $\Omega$ is a rectangle,
\begin{equation}
\label{embedding}
\|D^2 \mathrm{G}_\Omega f\|_{L^p(\Omega)^{2\times2}} \leq C(p) \|f\|_{L^p(\Omega) }, \qquad 2\leq p<\infty
\end{equation}
which further implies  the embedding $\cic{V}^{1,p}(\Omega) \hookrightarrow \cic{W}^{1,p}(\Omega) $ for all $2<p<\infty$ as well.  We outline how \eqref{embedding} follows from   the two bounds \begin{equation}\label{eq2}
\|T f\|_{L^2(\Omega)^{2\times2}} \leq C \|f\|_{L^2(\Omega) }, \qquad 
\|T f\|_{\mathrm{bmo}_r( \Omega)^{2\times2}} \leq C \|f\|_{ \mathrm{bmo}_z( \Omega)} \leq C  \|f\|_{ L^\infty( \Omega)}
\end{equation}
on the zeroth order linear operator $Tf=D^2 \mathrm{G}_\Omega f$, 
the first of which is the case $p=2$ of Theorem \ref{ercvx} and the second of which is Proposition \ref{bmolaw}. We define the sharp function  $$  f^\sharp (x):=\sup_{Q \ni x } \frac{1}{|Q|}\int_{Q} |f(y)-f_Q|\, \d y +\sup_{\substack{Q \ni x  \\ |Q| \geq 1}} \frac{1}{|Q|}\int_{Q} |f(y)|\, \d y  
$$
the supremum being taken over squares $Q \subset \Omega$ with axis-parallel sides.   It can be inferred from   \cite[Theorem 1.4]{Ch} that
\begin{equation}\label{eq3}
\|  f\|_{\mathrm{bmo}_r( \Omega)} = \|f^\sharp\|_{ L^\infty( \Omega) }.
\end{equation}
Moreover, for $1<p<\infty$
\begin{equation}\label{eq4}
\|f\|_{L^p(\Omega) } \leq C_1(p) \|f^\sharp\|_{L^p(\Omega)} \leq C_2(p) \|f \|_{L^p(\Omega)};
\end{equation}
the first half is a
local version (with no changes in the proof) of \cite[Theorem IV.2.2]{Stein}, and the second half is the pointwise bound $|f^\sharp| \leq 4 Mf$, where $M$ is the Hardy-Littlewood maximal function, followed by the maximal theorem.
The  bounds  in \eqref{eq2} then become
$$
\|(T f)^\sharp\|_{L^2(\Omega)^{2\times2}} \leq C \|f\|_{L^2(\Omega) }, \qquad 
\|(T f)^\sharp\|_{L^\infty(\Omega)^{2\times2}}    \leq C  \|f\|_{ L^\infty( \Omega)}
$$
and one applies Marcienkiewicz interpolation to the sublinear operator $f \mapsto (Tf)^\sharp$, and subsequently \eqref{eq4} again, to finally arrive at \eqref{embedding}.


\section{The two-dimensional Euler equations  in a bounded convex domain} In this section, we formulate problem \eqref{P} in a suitable weak sense and construct a weak solution. 
\label{s3}
\subsection{Weak formulation of \eqref{P} in bounded convex domains}
Our first task is to provide a weak formulation of \eqref{P} when $\Omega$ is a bounded convex domain, which we continue to assume throughout this section.
For convenience, we say that a triple of exponents $(r_1,r_2,r_3)$ is a H\"older triplet if
$$
 1\leq r_1,r_2,r_3 \leq  \infty, \quad \textstyle \frac{1}{r_1}+ \frac{1}{r_2} +\frac{1}{r_3}=1.
$$
\subsubsection{The trilinear form} We define the trilinear form
$$
b_{\Omega}(\u,\v,\cic{w}) = \int_{\Omega} (\mathbf{u}\cdot \nabla) \mathbf{v} \cdot  \cic{w} = ( (\mathbf{u}\cdot \nabla) \mathbf{v}, \cic{w})_{\Omega}, 
\qquad \u, \v, \w  \in  \D(\overline{\Omega}; \R^2).
$$ 
\begin{lemma} \label{trilinear1} {\emph{\cite{TEM2}}} Let $(r_1,r_2,r_3)$  be a H\"older triplet. Then the form $b_\Omega$ extends to a trilinear continuous form $$
b_\Omega: \cic{L}^{r_1}_{\tau}(\Omega) \times  \cic{W}^{1,r_2}  (\Omega) \times \cic{L}^{r_3}  (\Omega);
$$
moreover, we have
\begin{equation}
b_{\Omega}(\u,\v,\cic{w}) + b_{\Omega}(\u,\cic{w},\cic{v}) = 0, \qquad\forall \u \in \cic{L}^{r_1}_{\tau}(\Omega),  \quad \v, \w  \in   \cic{W}^{1,r_2}  (\Omega) \cap  \cic{L}^{r_3}  (\Omega).\label{tr-1}
\end{equation}
\end{lemma}
\begin{remark} \label{q431} This  remark, together with Remark \ref{q432} below,  motivates why we are mostly interested in the range    $p \in [\frac43,\infty\big)$. Define \begin{equation}
\label{sq} s(p)= \begin{cases} \frac{2p}{4-p}, & \frac43\leq p<2, \\ \mathrm{any }\; s \in [1,2) & p=2  \\ p & p>2  \end{cases}, \qquad z(p)= (s(p))'.\end{equation}
Assume $\u \in \cic{V}^{1,p}(\Omega)$. If $p\neq2$,  taking $r_1=p^*,r_2=p, r_3=z(p) $ in Lemma \ref{trilinear1}, and exploiting the   Sobolev embedding $\cic{V}^{1,p}(\Omega)\subset \cic{L}^{p^*}(\Omega)$ yields
\begin{equation} \label{convective}
\big\|\v\mapsto  b_\Omega(\u,\u,\v)\big\|_{\mathcal{L}( \cic{L}^{z(p) }(\Omega), \R)} \leq C_\Omega  \|\u\|_{\cic{V}^{1,p}(\Omega)}^2.
\end{equation}
By the Riesz representation theorem, moreover, it follows that  $(\cic{u}\cdot \nabla) \u \in \cic{L}^{s(p) }(\Omega)$ whenever $ p \in \big(\frac43,\infty\big)$. A similar argument gives the same result for any choice of $s(p) \in [1,2)$ when $p=$2.
 \end{remark}
Let $T>0$ be fixed. We will also make use of the functional
\begin{equation}
\label{tritime}
B_\Omega(\u,\v,\cic{w}) = \int_{0}^T b_\Omega (\u(t),\v(t),\cic{w}(t)) \, \d t
\end{equation}
initially defined on $\u,\v,\cic{w} \in \C^\infty([0,T] \times \overline{\Omega })$. In consequence of 
Lemma \ref{trilinear1} and H\"older's inequality, whenever $\{s_i\}_{i=1}^3,\{r_i\}_{i=1}^3$ are H\"older triplets, the estimates
\begin{equation}
|B_{\Omega}(\u,\v,\cic{w})|\leq \|\u\|_{L^{s_1}(0,T;\mathbf{L}^{r_1}_\tau(\Omega))}\|  \v\|_{L^{s_2}(0,T;\mathbf{W}^{r_2}(\Omega))} \|\cic{w}\|_{L^{s_3}(0,T;\mathbf{L}^{r_3}(\Omega))}\label{tr-time}
\end{equation}
hold for all $\u,\v,\cic{w} \in \C^\infty([0,T]\times \overline \Omega)$
 and thus $B_\Omega$ extends to a trilinear continuous form on the triple of spaces appearing in the right-hand side of \eqref{tr-time}. Analogously,
\begin{align} \nonumber
&B_{\Omega}(\u,\v,\cic{w})= -B_\Omega(\u,\cic{w},\v) \\ & \forall \u \in L^{s_1}(0,T;  \mathbf{L}^{r_1}_\tau(\Omega)), \, \v, \cic{w} \in { L^{s_2}(0,T;  \mathbf{W}^{1,r_2}(\Omega) )} \cap  L^{s_3}(0,T;\mathbf{L}^{r_3}(\Omega)) \label{tr-time2}
\end{align}

\subsubsection{Weak solutions}\label{DefWS} Let $p \in  \big[\frac43,\infty\big)$,
 $\u_0 \in \cic{V}^{1,p}(\Omega)$, ${\f} \in L^p(0,T;\cic{W}^{1,p}(\Omega))$ be given.
 A function $$\u:\Omega \times (0,T) \to \R^2$$ is a \emph{weak solution} to problem \eqref{P} if
\begin{align}
&\u \in L^\infty\big(0,T;\cic{V}^{1,p}(\Omega)\big) ; \label{DefWS1}\\
& \u \in \mathcal{C}\big([0,T];\cic{L}_\tau^{p}(\Omega)\big) \quad\textrm{and}\quad \cic{u}(0) = \u_0, \label{DefWS3} \\
\label{integralWS}
  &  -\int_0^T \big( \u (t), \v\psi'(t) \big)_{\Omega} \, \d t+ B_{\Omega}(\u,\u,\v \otimes \psi) =    \int_0^T\big({\f} (t),\v\psi(t)\big)_{ \Omega }\, \d t,
\end{align}
for all $\v \in \V$, and $   \psi \in \D(0,T).$
\begin{remark} Note that $\V$ is certainly not dense in $\cic{V}^{1,p} (\Omega)$. In fact, one can see that the closure of $\V$ in $\cic{V}^{1,p} (\Omega )$ is   $\cic{V}^{1,p} (\Omega)  \cap\cic{W}^{1,p}_0 (\Omega)$.
In a smooth domain $\Omega$, one can obtain  $\cic{V}^{1,p} (\Omega)$
as the closure of
$$
\W:=\{\v \in \C^\infty(\overline{\Omega}): \div \v =0 \textrm{  in  } \Omega, \, \v \cdot \cic{n}|_{\partial\Omega}=0\}=  \C^\infty(\overline{\Omega}) \cap \cic{V}^{1,p} (\Omega).
$$
This is not known to be true (and may very well be false) in a general convex domain. Ultimately, we are interested in testing the equation \eqref{integralWS} with functions of  $\cic{V}^{1,p} (\Omega)$, thus the choice of $\W$ as a test function space over $\V$ does not give any advantage. We explain how to circumvent this difficulty in Remark \ref{q432}.
\end{remark}
\begin{remark} \label{q432}Assume that $p \geq \frac43$ is such that the embedding $\cic{V}^{1,p}(\Omega)\hookrightarrow \cic{W}^{1,p}(\Omega)$ holds,  and $\u$ is a weak solution of \eqref{P}.
Note that each of the three functionals
$$
\v \mapsto \int_0^T \big( \u  , \v\psi'  \big)_{\Omega},  \quad \v \mapsto\int_0^T\big({\f}  ,\v\psi \big)_{ \Omega },  \quad \v \mapsto B_{\Omega}(\u,\u,\v \otimes \psi),
$$
is linear and continuous on $\cic{L}_\tau^{z(p)}(\Omega)$: the only nontrivial verification is  that, thanks to \eqref{convective} and the embedding 
\begin{align*}
|B_{\Omega}(\u,\u,\v \otimes \psi)| & \leq C_\Omega \|\u\|_{L^\infty(0,T;\cic{L}^{p}_\tau(\Omega))}  \|\u\|_{L^p(0,T;\cic{W}^{1,p} (\Omega))} \|\v \otimes \psi\|_{{L^{p'}(0,T;\cic{L}^{z(p)}_\tau(\Omega))} } \\ & \leq C \|\u\|_{L^\infty(0,T;\cic{V}^{1,p}(\Omega))}^2  \|\psi\|_{L^{p'}(0,T)}  \| \v\|_{\cic{L}_\tau^{z(p)}(\Omega)}.
\end{align*}
This shows that \eqref{integralWS} makes sense for  $\u $ as in \eqref{DefWS1}.
By density of $\V$ in $\cic{L}^{z(p)}_\tau(\Omega), $ the equality \eqref{integralWS} actually holds for all $\v \in \cic{L}^{z(p)}_\tau(\Omega) $.
Further observe that   $p \geq\frac32$ ensures the continuity of the embedding
$ \cic{W}^{1,p}(\Omega) \hookrightarrow
\cic{L}^{z(p)}(\Omega).
$
Hence, when $p\geq \frac32$, it is legitimate to take $\v \in  \cic{V}^{1,p}(\Omega) $ as a test function in \eqref{integralWS}.
\end{remark}
We summarize the construction of a weak solution to the Euler system \eqref{P} in the proposition below.\begin{proposition} \label{ThEx} Let $\Omega \subset \R^2$ be a bounded convex domain,    $p \in \big[\frac43,2\big] $, and let   $\u_0 \in \cic{V}^{1,p}(\Omega)$, ${\f} \in L^p(0,T;\cic{W}^{1,p}(\Omega))$ be given. 
Then Problem \eqref{P} has a weak solution $\u=\u(x,t)   \in L^{\infty}\big(0,T; \cic{V}^{1,p}(\Omega)\big)$ in the sense of  \eqref{DefWS1}-\eqref{integralWS}, with the further regularity
\begin{align}\label{Sol-regu}  
&\cic{u}  \in  \mathcal{C}\big([0,T]; \cic{L}^2_\tau(\Omega)\big)  , \qquad \u(t) \in   \cic{V}^{1,p}(\Omega)  \;\;\forall t \in [0,T],  \\ 
& \curl \u \in \mathcal{C}\big([0,T]; \mathrm{w}-{L}^p(\Omega)\big), \label{vort-reg}
\\
&  \label{Sol-reg} \pt \u \in  L^{p}\big( 0,T ; \cic{L}^{s(p)}(\Omega) ).
\end{align}
 \end{proposition}
 \begin{remark}Proposition \ref{ThEx} is a slight improvement of \cite[Proposition 7.1]{Tay}. {We also mention that the mere existence of a weak solution $\u  \in L^{\infty}\big(0,T; \cic{W}^{1,p}(\Omega)\big)$, without the additional regularity \eqref{Sol-regu}-\eqref{Sol-reg},  can be obtained as a consequence of \cite[Theorem 1]{GVA} and of the elliptic regularity Theorem \ref{ercvx}}. Our  proof follows the same general idea  (also employed in \cite{GVA})  of constructing  a weak solution to \eqref{P} on $\Omega$ as a limit of smooth solutions on a sequence of approximating domains.  Our modifications to the scheme  in \cite{Tay} allow us to extend the range of exponents from $\big(\frac{12}{7},2\big]$ to $\big[\frac43,2\big]$, as well as to obtain the additional in-time regularity of \eqref{Sol-regu}-\eqref{vort-reg}, which is not present in either \cite{Tay} or \cite{GVA}.  {Note that Proposition \ref{ThEx} will be completed in Theorem \ref{ThEx1}, where we re-introduce the pressure $\pi$, in the case $\frac43<p\leq2$.} \end{remark}
\subsection{Proof of Proposition \ref{ThEx}} \label{s4} \noindent
We assume throughout the proof that $\Omega$ is a bounded convex domain, $\frac43\leq p\leq 2$, and
\begin{equation}
\label{eqR}
  \u_0 \in {\cic{V}^{1,p}(\Omega)}, \; {\f} \in L^p(0,T; {\cic{W}^{1,p}(\Omega)}), \quad
  R:=\Big(\|\u_0\|^p_{{\cic{V}^{1,p}(\Omega)}}+\|{\f}\|_{ L^p(0,T; {\cic{W}^{1,p}(\Omega))}}^p\Big)^{\frac1p}  .
\end{equation} 
The symbol  $\Q$ will stand for a positive increasing function of its argument, depending only on $p$, $T$ and  $\Omega$, which can  be explicitly computed 
and is allowed to vary from line to line. Also, all the constants implied by the almost inequality signs $(\lesssim)$ are allowed to depend on $p$ and $\Omega$ without 
explicit mention.
 The proof will proceed through several steps.
\subsubsection{Approximation of the domain and of the data} \label{ss41p} We approximate $\Omega$ by a   sequence of smooth convex domains $\Omega_n$ as in 
 \eqref{gammacv}. As discussed in Section 2, the implicit constants in the Sobolev embeddings for  $\Omega_n$, as well as the $\cic{L}^{p}(\Omega_n)$ 
 norm of the Leray projector $P_{\Omega_n}$ are uniform in $n$ and depend only on  (the Lipschitz character of)   $\Omega$. Our use of almost-inequalities 
 and of the function $\Q(\cdot)$ will reflect this uniformity.

  We turn to the approximation of the data in \eqref{P}.
Let  $\omega_0=\curl \u_0$. We set
 \begin{align} &
\omega_0^{(n)}(x)=  \Big(\omega_{0} \cic{1}_{\Omega_{n-1}} \Big)*\phi^{2}_{\eps_n}  (x),\qquad x \in \Omega, \label{approxdata}\\ & {\f}^{(n)}(x,t) = \Big(\big(  {\f}  \cic{1}_{\Omega_{n-1}x}  \cic{1}_{(\eps_n, T-\eps_n)t} \Big)*\big(\phi^{2}_{\eps_n}\otimes \phi^{1}_{\eps_n}   \big)(x,t), \quad x \in \Omega, \, t \in [0,T],
\end{align}
with $\eps_n>0$, $\eps_n \to 0$ chosen small enough to have \begin{equation} \label{omegasupport}\supp \omega_0^{(n)}\Subset \Omega_n ,\qquad \supp {\f}^{(n)}  \Subset \Omega_n \times (0,T).
\end{equation} Note that
 \begin{align}\label{convinomega} & \omega_0^{(n)} \in \D(\Omega_n), \qquad
\|\omega_0^{(n)}\|_{ L^p(\Omega)}\leq \|\omega_0 \|_{ L^p(\Omega)},\qquad \widetilde{ \omega_0^{(n)} }\to \omega_0 \; \textrm{in}\; L^p(\Omega), \\ &
{\f}^{(n)}  \in \D(\Omega_n \times( 0,T)),  \quad
\|{\f}^{(n)}\|_{ L^p(0,T; \cic{W}^{1,p}(\Omega_n))}\leq\|  {\f} \|_{ L^p(0,T; \cic{W}^{1,p}(\Omega))},\\ &\|\f^{(n)}-\f\|_{ L^p(0,T; \cic{W}^{1,p}(\Omega_n))} \to 0, \qquad n \to \infty. \label{convinf}
\end{align}
\begin{lemma} Define
 \label{lemmainitialu}
$$ \u_0^{(n)}: \Omega_n \to \R^2, \qquad
 \u_0^{(n)} = \mathbf{K}_{\Omega_n} \omega_0^{(n)}.
 $$
 We then have
 \begin{align}\label{smoothid1} &\u^{(n)}_0 \in \C^\infty(\overline{\Omega}_n, \R^2),  \\
&  \nabla \cdot  \u^{(n)}_0 =0  \;\textrm{ \emph{in} }\;  \Omega_n, \quad \u^{(n)}_0 \cdot \cic{n}|_{\partial \Omega_n}=0;  \label{smoothid2} \\
 &
\|\u_0^{(n)}\|_{ \cic{W}^{1,p}(\Omega_n)}\leq \Q(R).  \label{convinf1}
\end{align}
\end{lemma}
\begin{proof}Immediate from \eqref{convinomega} and Theorem \ref{ercvx}.
\end{proof}

\subsubsection{The approximating problems and \emph{a priori} estimates} \label{ss42p}
We   consider the following evolution problem:
\begin{equation} \label{Pn}  \tag{P$_n$}
\begin{cases}
\pt \cic{u}(x,t)  + (\u \cdot \nabla) \cic{u}(x,t)+ \nabla \pi(x,t) = {\f}^ {(n)}(x,t), & x\in \Omega_n, t \in (0,T) \\
\nabla \cdot \cic{u} (x,t) =0 & x\in \Omega_n, t \in (0,T) \\
\u\cdot \mathbf{n} (x,t)=0, &  x\in \partial\Omega_n, t \in (0,T)\\
\u(x,0)=\u_0^{(n)}(x) & x \in \Omega_n
\end{cases}
\end{equation}
Due to the smoothness  \eqref{smoothid1} of the data $\cic{u}_0^{(n)}, {\f}^{(n)}$ and of the domain $\Omega_n$, the classical result of  Kato 
\cite{Kato} (see also \cite{Bardos,Bardos2}) yields the existence of
 a unique classical solution of \eqref{Pn}:  $$(\u^{(n)},\pi^{(n)}) \in \C^{\infty}(\overline{\Omega_n} \times [0,T] )^2 \times \C^{\infty}(\overline{\Omega_n} 
 \times [0,T]).$$
 Taking $\curl$ of the equation in \eqref{Pn}, we obtain the following equation for $\omega^{(n)}$:
\begin{equation} \label{vortPn}
\begin{cases}
\pt \omega^{(n)}(x,t)  + (\u^{(n)} \cdot \nabla) \omega^{(n)}(x,t)  = g^{(n)}(x,t), & x\in \Omega_n, t \in (0,T), \\
\omega^{(n)}(x,0)=\omega_0^{(n)}(x) & x \in \Omega_n,
\end{cases}
\end{equation}
where $g^{(n)}= \curl{\f}^{(n)}.$
We now derive \emph{a priori} estimates  on the smooth solutions $\cic{u}^{(n)}$ via the fundamental estimate on the vorticity $\omega^{(n)} $ in the lemma below.
\begin{lemma} \label{apvort} Let $R$ be as in \eqref{eqR}.
The following estimate holds:
$$
  \|\omega^{(n)} \| _{L^\infty(0,T;L^{p}(\Omega_n))}    \leq  e^{T} R.
$$
\end{lemma}
\begin{proof}
Multiplying \eqref{vortPn}by $p \omega^{(n)} |\omega^{(n)} |^{p-2}$, and integrating on $\Omega_n$, we obtain 
\begin{align*}
\ddt \|\omega^{(n)}(t)\|_{L^p(\Omega_n)}^p &= p (g^{(n)}(t), \omega^{(n)} |\omega^{(n)} |^{p-2} )_{\Omega_n} \leq \|  g^{(n)}(t)\|_{L^p(\Omega_n)}^p + (p-1)\|\omega^{(n)}(t)\|_{L^p(\Omega_n)}^p \\ &\leq   \|  g^{(n)}(t)\|_{L^p(\Omega_n)}^p +  (p-1)\|\omega^{(n)}(t)\|_{L^p(\Omega_n)}^p,
\end{align*}
so that the desired conclusion  follows from the Gronwall lemma and \eqref{convinomega}-\eqref{convinf}. We used that
$$p\big((\u^{(n)} \cdot \nabla) \omega^{(n)}  ,\omega^{(n)} |\omega^{(n)} |^{p-2} \big)_{\Omega_n} = \big(\u^{(n)}, \nabla|\omega^{(n)}|^p \big)_{\Omega_n}=0
$$
which stems from $\u^{(n)}$ being divergence free and with zero normal component on $\partial \Omega_n$.
\end{proof}  
Via an application of \eqref{K-boundstrong}, Lemma \ref{apvort}  entails the following a priori estimate on the solutions $\u^{(n)}$:
\begin{equation}
\label{preliminaryu}
  \|\u^{(n)} \|_{L^\infty(0,T;\cic{V}^{1,p}(\Omega_n))}    \leq \Q(R).
\end{equation}
We then derive an equicontinuity result for the vorticity.
\begin{lemma} \label{eqvort}
For each $\psi \in L^{p'}(\Omega), $ the sequence of  functions
$$
t \in [0,T] \mapsto \big( \widetilde {\omega^{(n)}}(t), \psi\big)_{\Omega},
$$
is equicontinuous on $[0,T]$.
\end{lemma}
\begin{proof}  Let $\eps>0$ and $\psi \in L^{p'}(\Omega)$ be given; we have to show that there exist $\delta= \delta(R,\eps, \psi), N=N(\eps,\psi)$ such that
\begin{equation}
  |t_2-t_1| < \delta \implies \sup_{n \geq N}\big|  \big( \widetilde {\omega^{(n)}}(t_2), \psi\big)_{\Omega}
 - \widetilde {\omega^{(n)}}(t_1), \psi\big)_{\Omega}\big| <\eps.
\label{equicontinuity}
\end{equation}
We first assume  $\psi=\varphi \in \D(\Omega)$. Take $N=N(\varphi)$ large enough so that $\supp \varphi \Subset \Omega_N$.  Then, we multiply by $\varphi $ 
the  equation in \eqref{vortPn} and  integrate over $\Omega_n$: observing that
$$
\int_{\Omega_n}( \u^{(n)} \cdot \nabla \omega^{(n)} \varphi )+ \int_{\Omega_n} (\u^{(n)} \cdot \nabla  \varphi) \omega^{(n)} = - \int_{\Omega_n} (\div{\u^{(n)}}  \nabla  \varphi) \omega^{(n)} - \int_{\partial \Omega_n}( \varphi \omega^{(n)}) \cic{u}^{(n)} \cdot \cic{n}=0,
$$
since $\u^{(n)}$ is divergence free and has zero normal component on $\partial \Omega_n$, and integrating on $(t_1,t_2)$, one finds
\begin{align*} &\quad \big|  \big( \widetilde {\omega^{(n)}}(t_2), \varphi\big)_{\Omega} - \widetilde {\omega^{(n)}}(t_1), \varphi\big)_{\Omega}\big| =
\big|  \big(  {\omega^{(n)}}(t_2), \psi\big)_{\Omega_n} -   {\omega^{(n)}}(t_1), \psi\big)_{\Omega_n}\big| \\ & \leq \Big| \int_{t_1}^{t_2} \big(   \u^{(n)} \cdot \nabla  \varphi,  \omega^{(n)}\big)_{\Omega_n} \Big| + \Big| \int_{t_1}^{t_2} \big(g^{(n)}, \varphi\big)_{\Omega_n} \Big|\\
& \leq |t_1 - t_2| \|\u^{(n)}\|_{L^\infty(0,T; \cic{L}^{p'}(\Omega_n))}\|\nabla\varphi\|_{\cic{L}^\infty(\Omega_n)} \|\omega^{(n)}\|_{L^\infty(0,T;  {L}^{p}(\Omega_n))} \\ & \quad+ |t_1-t_2|^{\frac{1}{p'}}\|g^{(n)}\|_{L^p(0,T; L^p(\Omega_n))}\|\varphi\|_{{L}^{p'}(\Omega_n)} \\ & \leq T^{\frac1p}|t_1-t_2|^{p'} \Big(  \|\u^{(n)}\|_{L^\infty(0,T; \cic{V}^{1,p}(\Omega_n))}\|\nabla\varphi\|_{\cic{L}^\infty(\Omega_n)} \|\omega^{(n)}\|_{L^\infty(0,T;  {L}^{p}(\Omega_n))} \\ & \quad + \|g^{(n)}\|_{L^p(0,T; L^p(\Omega_n))}\|\varphi\|_{{L}^{p'}(\Omega_n)} \Big)
\\ & \leq T^{\frac1p}|t_1-t_2|^{p'} \Q(R)\big( \|\nabla\varphi\|_{\cic{L}^\infty(\Omega_n)} +\|\varphi\|_{{L}^{p'}(\Omega_n)} \big).
\end{align*} 
We used the Sobolev embedding  $\cic{V}^{1,p}(\Omega_n) \hookrightarrow \cic{L}^{p'}(\Omega_n)$, which holds in the region $p^* \geq p'$, that is $p\geq \frac43$, and \eqref{preliminaryu}, as well as Lemma  
 \ref{apvort}. Now,  for $\psi    \in L^{p'}(\Omega) $, take $\varphi \in \D(\Omega)$ with $\|\varphi-\psi\|_{ L^{p'}(\Omega)}<\eps/4\|\omega^{(n)}\|_{L^\infty(0,T; L^p(\Omega))}$. Set
$$
Q(\eps,\psi,R):= T^{\frac1p}\big( \|\nabla\varphi\|_{L^\infty(\Omega)} +\|\varphi\|_{{L}^{p'}(\Omega )} \big).
$$ Then for $n \geq N(\varphi)= N(\eps, \psi),$
\begin{align*}
\big|\big( \widetilde {\omega^{(n)}}(t_2)- \widetilde {\omega^{(n)}}(t_1), \psi\big)_{\Omega} \big| & \leq \big|\big( \widetilde {\omega^{(n)}}(t_2)- \widetilde {\omega^{(n)}}(t_1) , \varphi\big)_{\Omega}\big| + \big|\big( \widetilde {\omega^{(n)}}(t_1) - \widetilde {\omega^{(n)}}(t_2)  , \varphi-\psi\big)_{\Omega}\big|  \\
&  \leq |t_1-t_2|^{p'} Q(\eps,\psi,R) + 2 \|\omega^{(n)}\|_{L^\infty(0,T; L^p(\Omega_n))}\|\varphi-\psi\|_{ L^{p'}(\Omega)}
 <\eps,
\end{align*}
if $|t_1-t_2| \leq \delta:= { \big(\frac{\eps}{2} Q(\eps,\psi,R)\big)}^{\frac1{p'}}.$ Thus \eqref{equicontinuity} is fulfilled.
\end{proof} 
\noindent
Denote $\Psi^{(n)}:= \mathrm{G}_{\Omega_n} \omega^{(n)}\ $. In view of Lemma \ref{apvort},  the  elliptic regularity   \eqref{epest1} gives
\begin{align} \label{energy-psi1}
&\Psi^{n}(t) \in  W^{2,p}(\Omega_n) \cap  W^{1,p}_0(\Omega_n) \qquad  \textrm{a.e.\ }t \in (0,T), \\ &  \label{energy-psi2}
   \|  \Psi^{(n)}   \|_{L^\infty(0,T; W^{2,p}(\Omega_n))} \lesssim  \|   \omega^{(n)}  \|_{L^\infty(0,T; L^p(\Omega_n))} \leq \Q(R).
\end{align}
Note     that the embedding $W^{2,p}(\Omega_n) \cap W^{1,p}_0(\Omega_n) \hookrightarrow W^{1,r}_0(\Omega_n) $ is continuous (with uniformity in $n$, see Section \ref{s2}) 
whenever $1\leq r< p^*$. Thus, we use the extension property \eqref{extensionW} to derive from   \eqref{energy-psi1}-\eqref{energy-psi1} that
\begin{equation}
\label{w1r}
\|\widetilde{\Psi^{(n)}}\|_{L^\infty(0,T; W^{1,r}_0(\Omega ))}= \| \Psi^{(n)}\|_{L^\infty(0,T; W^{1,r}_0(\Omega_n))} \leq \Q(R), \qquad 1\leq r<p^*.
\end{equation}
 We will actually use \eqref{w1r} with $r=p$ and $r=2$.   Lastly,
we derive a uniform estimate on the time-derivative  of $ \u^{(n)}$.
\begin{lemma} \label{apder}
Let $
 s(p)$ be as in \eqref{sq}.  We have the estimate
\begin{equation} \label{energy-dtpsi}
 \| \pt \widetilde{\Psi^{(n)}}\|_{L^p(0,T; W^{1,s(p)}_0(\Omega))}    \leq \Q(R).
\end{equation}
\end{lemma}
\begin{proof}
We set $$\w^{(n)}= \nabla \Psi^{(n)}=(\u^{(n)})^\perp.$$
  Observe that, since $\w^{(n)}$ is a gradient, $P^\perp_{\Omega_n} \pt \w^{(n)} = \pt P^\perp_{\Omega_n}  \w^{(n)}= \pt \w^{(n)} $. At this point,
note that, for any $\varphi \in \D(\Omega_n)$,
\begin{align*}
\big(\pt \w^{(n)},\nabla \varphi \big)_{\Omega_n} &= \ddt \big( \mathbf{w}^{(n)},\nabla \varphi \big)_{\Omega_n} =   -\ddt \big(  \Delta \Psi^{(n)}  ,  \varphi \big)_{\Omega_n}  = - \big( \pt  \omega^{(n)}  ,  \varphi \big)_{\Omega_n} \\ & =   \big( \u^{(n)} \cdot \nabla \omega^{(n)}  , \varphi \big)_{\Omega_n} -
 \big( g^{(n)}  , \varphi \big)_{\Omega_n} \\
& = \big(-\omega^{(n)} \u^{(n)}   +   ({\f}^{(n)})^\perp   , \nabla \varphi \big)_{\Omega_n}.
\end{align*}
We integrated by parts the first term in the next to last line to arrive at the last line.
 Thus, we obtain the equation
\begin{equation}
\label{eqw}\pt \w^{(n)} = -P^\perp_{\Omega_n}\big[-\omega^{(n)} \u^{(n)} + ({\f}^{(n)})^\perp\big], \qquad \textrm{in }\, \Omega_n\times (0,T).
\end{equation} Observe that, by Sobolev embeddings and   \eqref{energy-psi2},
\begin{equation}
\label{boundun}\|\u^{(n)}\|_{L^\infty(0,T;\cic{L}^{r}(\Omega_n)) }  \lesssim \|\Psi^{(n)}\|_{L^\infty(0,T;W^{2,p}(\Omega_n))} \lesssim \Q(R), \quad
r= \begin{cases} p^*, & \frac43\leq p<2, \\  s(p)^* & p=2 .\end{cases}
\end{equation}
Theorem \ref{helmthm} ensures that  $P_{\Omega_n}^\perp$ is a linear bounded operator on $\cic{L}^s(\Omega)$, with bound independent on $n$. This, together with   H\"older's inequality and Sobolev embeddings gives
\begin{align}
\label{fill}
\big\|P^\perp_{\Omega_n}\big[\omega^{(n)} \u^{(n)}\big]\big\| _{L^\infty(0,T;\cic{L}^{s(p)}(\Omega_n))}&\lesssim    \big\|\omega^{(n)} \u^{(n)} \big\|_{L^\infty(0,T;\cic{L}^{s(p)}(\Omega_n))} \\ &\leq    \|\omega^{(n)}\|_{L^\infty(0,T;L^p(\Omega_n))} \|\u^{(n)}\|_{L^\infty(0,T;\cic{L}^{r}(\Omega_n))} \nonumber
\\ & \leq \Q(R). \nonumber
\end{align} To obtain the last inequality we used Lemma   \ref{apvort} and \eqref{boundun}.
Since we also have
$$
\|P^\perp_{\Omega_n}({\f}^{(n)})^\perp\|_{L^{p}(0,T;\cic{L}^s(\Omega_n))} \lesssim \|({\f}^{(n)})^\perp\|_{L^{s}(0,T;\cic{L}^p(\Omega_n))} \lesssim \| {\f} \|_{L^{p}(0,T;\cic{W}^{1,p}(\Omega ))} \leq \Q(R),
$$
we obtain, comparing \eqref{eqw} with \eqref{fill} and the last display, that
$$
\|\pt \cic{u}^{(n)}\|_{L^p((0,T), \cic{L}^{s(p)}(\Omega_n))} =\|\pt \cic{w}^{(n)}\|_{L^p((0,T), \cic{L}^{s(p)}(\Omega_n))} \leq \Q(R),
$$
Observing that $\widetilde{\partial_t \u^{(n)}} = \partial_t \widetilde{\u^{(n)}}  $ in $\Omega\times [0,T]$ and the   extension by zero  to $\Omega $ preserves the $L^p$ norms, the above display yields
$$
 \|\pt \widetilde{\cic{u}^{(n)}}\|_{L^p((0,T), \cic{L}^{s(p)}(\Omega  ))} \leq \Q(R)
$$
and the 
 conclusion of the Lemma then follows from  the Poincar\'e inequality.
\end{proof} \noindent
\subsubsection{Conclusion of the proof.} \label{ss43p}
We will proceed in steps.
\vskip1mm \noindent\textsc{Step 1. Compactness. }In view of Lemma \ref{apvort},  the sequence of functions $t \mapsto \widetilde{\omega^{(n)}(t)}$ takes values in the closed ball $B$ of  $L^p(\Omega)$ of radius $\e^T R$. We recall that, due to  reflexivity and separability of $L^p(\Omega)$, the induced $\textrm{w}-L^{p}(\Omega)$ topology on $B$ is metrizable by $$
 d(u,v)= \sum_{k\geq 0} 2^{-k} \frac{|(u-v,\psi_k)_{\Omega}|}{1+|(u-v,\psi_k)_{\Omega}|}
 $$
where $\{\psi_k\}$ is a dense sequence in $L^{p'}(\Omega)$,   and $(B,d)$  is a compact metric space. The conclusion of Lemma \ref{eqvort} easily implies that the sequence $
 \widetilde{\omega^{(n)}}$ is equicontinuous on $[0,T]$ with values in the compact metric space $(B,d)$. Therefore, an application of Ascoli-Arzel\`a's theorem yields that
the sequence $\{ \widetilde{\omega^{(n)}}\}$ is precompact in $\C([0, T]; \mathrm{w}-L^{p}(\Omega))$, so that up to a subsequence
\begin{equation}
\label{omegaconv-l2weak}
 \widetilde{\omega^{(n)}} \to \omega  \; \textrm{ in } \;\C([0, T]; \mathrm{w}-L^{p}(\Omega)).
\end{equation}
We derive some consequences from \eqref{omegaconv-l2weak}.
\begin{lemma} \label{lemmaomega} We have that
\begin{align}
& \omega \in  \C([0, T]; \mathrm{w}-L^{p}(\Omega)) \cap L^\infty(0, T;  L^{p}(\Omega)) \cap \C([0, T]; H^{-1}(\Omega)), \label{c0TH-1} \\ & \label{initialdata-omega}
\omega(0)= \omega_0 \in L^p(\Omega).
\\
&   \label{omegaconv-a.e}
 \omega^{(n)}(t) \to \omega(t)  \quad \textrm{ \emph{in} } \; H^{-1}(\Omega) \qquad \forall t \in [0,T].
 \end{align}
\end{lemma}
\begin{proof} The first two inclusions in \eqref{c0TH-1} have just been shown. Actually the first inclusion implies the  stronger property 
\begin{equation}  
\omega(t)   \in  L^p(\Omega)  \qquad \forall t \in [0,T]. 
\label{regularomega}
\end{equation}
The third follows from the first and the  continuous embedding $\mathrm{w}-L^{p}(\Omega) \hookrightarrow H^{-1}(\Omega)
.$\footnote{The continuity of the embedding $\mathrm{w}-L^{p}(\Omega) \hookrightarrow H^{-1}(\Omega)$ for $1<p<\infty$ is a consequence 
of  the compactness of $L^p(\Omega) \hookrightarrow H^{-1}(\Omega)$. This in turn follows from the adjoint compact embedding $H^{1}_0(\Omega) \hookrightarrow L^{p'}(\Omega)$.} 
Then, \eqref{omegaconv-l2weak}    yields in particular  that
$\omega(0)$ is the $\mathrm{w}-L^p(\Omega)$-limit of $\{ \widetilde{ \omega^{(n)} (0)}\}$;
but, in view of \eqref{convinomega}, this implies
$\omega(0)= \omega_0 \in L^p(\Omega)$, so that \eqref{initialdata-omega} follows.
\noindent
Finally, \eqref{omegaconv-a.e}  follows  from \eqref{omegaconv-l2weak} and the continuity of the embedding $\mathrm{w}-L^p(\Omega) \hookrightarrow H^{-1}(\Omega)$.
\end{proof}
Define
$$
\Psi: \Omega \times [0,T] \to \R,
\qquad \Psi:= \mathrm{G}_\Omega \omega.$$ Then,
an application of \eqref{ercvx} and \eqref{c0TH-1}   yields
\begin{equation}
\label{boundPsi}
\Psi(t) \in W^{2,p} \cap W^{1,p}_0(\Omega), \; \|\Psi(t)\|_{ W^{2,p}(\Omega) } \lesssim \|\omega\|_{L^\infty(0,T; L^p(\Omega))}   \leq \Q(R), \quad\forall t \in [0,T].
\end{equation} 
\begin{lemma}[Convergence of the stream functions]   \label{}
We have that
\begin{equation} \label{weakpsi}
\widetilde{\Psi^{(n)}} \to \Psi:= \mathrm{G}_\Omega \omega \qquad \textrm{ \emph{in} }\; L^2(0,T;  H^1_0(\Omega)).
\end{equation}
\end{lemma}
\begin{proof} 
We begin by observing that, in view of Lemma \ref{lemmaconvcvx},
  the convergence \eqref{omegaconv-a.e} implies that  $\widetilde{\Psi^{(n)}} (t) \to \Psi(t)$ in $ H^1_0(\Omega)$ for all $t\in[0,T]$.  We also read from  \eqref{w1r} with $r=2$  that $$\sup_{n} \sup_{t \in[0,T]} \big\|\widetilde{\Psi^{(n)}}(t) \big\|_{ H^{1}_0(\Omega)} \leq \Q(R),$$ so that 
the claim of the lemma follows by the dominated convergence theorem. 
\end{proof} 
The final observation of Step 1 is that \eqref{weakpsi} and the bound \eqref{energy-dtpsi} guarantee that $\{\widetilde{\Psi^{(n)}}\}$ is a uniformly bounded sequence in $W^{1,p}\big(0,T; W_0^{1,  {s(p)}}(\Omega) \big)$, whence \begin{equation} \pt \Psi \in L^{p}(0,T; W_0^{1,  {s(p)}}(\Omega) ) .  \label{strongpsidt}
\end{equation}
\vskip1mm \noindent\textsc{Step 2. Construction of the solution to \eqref{P}.}
We now introduce $$
\u= \u(x,t): \Omega\to \R^2, \qquad \u(t) :=\mathbf{K}_\Omega \omega(t), \; t \in [0,T].$$
As a consequence  of \eqref{K-boundstrong}, Lemma \ref{lemmaomega}, and \eqref{strongpsidt}    we have
\begin{align}
 &\u(t) \in   \cic{W}^{1,p}_\tau(\Omega)  \qquad \forall t \in [0,T], \label{regularu11} \\
&
\|\u\|_{\cic{L}^\infty(0,T; \cic{W}^{1,p}_\tau(\Omega))}\lesssim    \|\omega\|_{L^\infty(0,T; L^p(\Omega))} \leq \Q(R), \label{regularu1} \\
 &\u \in \C([0,T]; \cic{L}^2(\Omega)), \label{regularu2} \\ & \pt \u \in  L^{p}(0, T; \cic{L}^{s(p)} (\Omega) ),  \label{regulardt} \\
 \label{initialdata-u} &  \u(0) = \cic{\cic{L}}^2-\lim_{t=0} \u(t)= \mathbf{K}_\Omega \omega_0 = \u_0;
\end{align}
furthermore, the convergence  \eqref{weakpsi}  translates into 
\begin{equation} \label{strconv-u}
\widetilde{\u^{(n)}} \to \u \qquad \textrm{in} \;L^2(0,T;\mathbf{L}^2(\Omega)).
\end{equation}
\vskip1mm \noindent\textsc{Step 3. Passage to the limit.} Thanks to \eqref{regularu1}-\eqref{initialdata-u} we see that \eqref{DefWS1}-\eqref{DefWS3} and \eqref{Sol-reg} hold true.
 To show that $\u$ is a weak solution to \eqref{P} we are left with proving that the distributional equality \eqref{integralWS} holds
for any $\v \in \V, \psi \in \D(0,T).$

  Multiplying and integrating by parts in \eqref{Pn} leads to the equation
\begin{align}
\label{convergence-2}
 &  -\int_0^T \big( \u^{(n)}(t), \v\psi'(t) \big)_{\Omega_n} \, \d t+ B_{\Omega_n} (\u^{(n)}, \u^{(n)},\v\otimes\psi) = \int_0^T\big({\f} (t),\v\psi(t)\big)_{ \Omega_n } \, \d t.
\end{align}
Due to \eqref{strconv-u} and \eqref{convinf} respectively, it is immediate to see  that the first term on the left hand side, and the right hand side of \eqref{convergence-2} converge to the homologous terms in \eqref{integralWS}. We now treat the nonlinear term. We have that, using \eqref{tr-time2}, and subsequently \eqref{tr-time} with $s_1=r_1=s_3=r_3=2,$ $s_2=r_2=\infty$,
\begin{align*} &
 \quad \Big|B_{\Omega_n}  (\u^{(n)} , \u^{(n)} ,\v\psi )  -  B_{\Omega} ((\u  , \u ,\v\psi)  \Big| =  \Big|B_{\Omega} (\u^{(n)} , \v\psi, \u^{(n)}  )   -  B_{\Omega} (\u  , \v\psi,\u)   \Big|\\ &\leq\Big|B_\Omega(\u-\widetilde{\u^{(n)}}  , \v\psi, \u^{(n)}) \Big| +\Big| B_\Omega(\u,\v \psi, \u-\widetilde{\u^{(n)}}  )\, \d t  \Big|  \\ & \lesssim
 \|\u-\widetilde{\u^{(n)}}\|_{L^2(0,T;\mathbf{L}^2(\Omega))}\big(\|\widetilde{\u^{(n)}}\|_{L^2(0,T;\mathbf{L}^2(\Omega))}+\|\u \|_{L^2(0,T;\mathbf{L}^2(\Omega))}   \big) \|\psi\|_{L^\infty(0,T)}\|\nabla \v\|_{L^\infty(\Omega)},
\end{align*}
so that \eqref{strconv-u} allows us to conclude that \eqref{integralWS} holds. Note that $\{\widetilde {u^{(n)}}\}$ is a bounded sequence in $L^2(0,T;\mathbf{L}^2(\Omega))$. Also note that for $n$ sufficiently large, $\mathrm{supp}\, \v \subset {\Omega_n}$, thus we could replace $B_{\Omega_n}$ by $B_\Omega$ in the second step. The proof of Proposition \ref{ercvx} is therefore complete.

\section{Main Results} \label{s5}
Our first main result is an improvement of Proposition \ref{ThEx}:  we exploit the extra regularity \eqref{Sol-reg} of the weak solution to Problem \eqref{P} given in Proposition \ref{ThEx} to recover the pressure $\pi$ and show that the pair $(\u,\pi)$ thus obtained satisfies \eqref{P} almost everywhere in $\Omega \times (0,T).$ We summarize the properties of the weak solutions obtained in the main results, which we term  (S)$_p$\emph{-weak solutions}, in the following definition.
\begin{definition}  \label{Defsws} Let $\frac43<p< \infty$, $\u_0 \in \cic{V}^{1,p}(\Omega), \; {\f} \in L^{p}(0,T; \cic{W}^{1,p}(\Omega))$. Denote also $q=\max\{2,p\}$.
We call (S)$_p$\emph{-weak solution} to the Euler system \eqref{P} with data $\u_0,\f$ a  pair $$(\u:\Omega\times [0,T] \to \R^2,\pi:\Omega\times(0,T) \to \R)$$ with $ \u(0)= \u_0,$ with the properties
\begin{align}
& \cic{u} \in \mathcal{C}\big([0,T]; \cic{L}^q_\tau(\Omega)\big)\cap L^{\infty}\big(0,T; \cic{V}^{1,p}(\Omega)\big), \quad \u(t) \in   \mathbf{V}^{1,p} (\Omega) \;\;\forall t \in[0,T],  \label{thex3pf2} \\
& \curl \u \in   \mathcal{C}\big([0,T]; \mathrm{w}-{L}^p(\Omega)\big)   \label{vort-reg1} \\
&
 \pt \u, \; \nabla\pi \in  L^{p}\big( 0,T ; \cic{L}^{s(p)}(\Omega) ),   \label{thex3pf3} \\
&   \label{PLq}
\pt \u + (\u \cdot \nabla) \u +\nabla \pi = {\f} \qquad \textrm{ in } \;\; L^{p}\big(0,T;  \cic{L}^{s(p)}(\Omega)\big) \textrm{ and a.e. in }\Omega \times (0,T) .  
\end{align}
Multiplying \eqref{PLq} by an appropriate test function and integrating on $\Omega \times (0,T)$, it is easy to see that an (S)$_p$-weak solution to \eqref{P} is a weak solution in the sense of Subsection \ref{DefWS}. 
\end{definition}
\begin{theorem} \label{ThEx1} Let $\Omega \subset \R^2$ be a bounded convex domain and
  $p \in \big(\frac43,2\big] $. Given $$\u_0 \in \cic{V}^{1,p}(\Omega), \qquad {\f} \in L^{p}(0,T; \cic{W}^{1,p}(\Omega)),$$
there exists a pair $(\u:\Omega\times [0,T] \to \R^2,\pi:\Omega\times[0,T] \to \R)$ which is an \emph{(S)}$_p$-weak solution to the Euler system \eqref{P} with data $\u_0,\f$.
 \end{theorem}

\begin{proof}[Proof of Theorem \ref{ThEx1}]
Assume $p\in (\frac43,2]$ and let $\u$ be the solution to Problem \eqref{P} given by Proposition \ref{ThEx}. Note that, in view of   Remark
\ref{q431}
the functional $\v \mapsto B_{\Omega}(\u,\u,\v)$ extends boundedly to $\cic{L}^{p'}(0,T; \cic{L}^{z(p)}_\tau(\Omega))$, the function $(\u \cdot \nabla ) \u  \in L^\infty(0,T; \cic{L}^{s(p)}(\Omega)),$ and we have the equality
\begin{equation} \label{udotgrad}
B_{\Omega}(\u,\u,\v)= \int_0^T \big( (\u(t) \cdot \nabla ) \u(t), \v(t) \big)_{\Omega} \, \d t, \qquad  \forall \v \in L^{p'}(0,T; \cic{L}^{z(p)}_\tau(\Omega)),
\end{equation}
by density of $\V \otimes \D(0,T) $ in  ${L}^{p'}(0,T; \cic{L}^{z(p)}_\tau(\Omega))$.
Due to \eqref{Sol-reg}, we also have  $\pt \u \in  L^{p}\big( 0,T ; \cic{L}^{s(p)}(\Omega) )$, and the equality
$$
\int_0^T \big( \pt  \u  , \v\psi  \big)_{\Omega}= -\int_0^T \big( \u  , \v\psi'  \big)_{\Omega} = - B_{\Omega}(\u,\u,\v\otimes \psi) + \int_0^T \big( {\f},\v \psi\big),   
$$
valid for all $\v \otimes \psi \in \V \otimes \D(0,T),$ carries over by density to
\begin{equation} \label{annikil}
\int_0^T \Big( \pt  \u(t) + (\u(t) \cdot \nabla)\u(t)-{\f}(t) , \v(t)  \Big)_{\Omega}   \, \d t =0, \qquad \forall \v \in\cic{L}^{p'}(0,T; \cic{L}^{z(p)}_\tau(\Omega)).
\end{equation}
 Now we define
\begin{equation} \label{gradp}
\cic{w}: \Omega \times (0,T) \to \R^2, \qquad \cic{w} := \pt \u + (\u \cdot \nabla ) \u  - {\f};
\end{equation}
we have
$
\cic{w} \in L^{p}(0,T; \cic{L}^{s(p)}(\Omega)),
$
and we read from \eqref{annikil} that  $\cic{w}$ belongs to the annihilator $X_{p'}$ of $\cic{L}^{p'}(0,T; \cic{L}^{z(p)}_\tau(\Omega))$ in $\cic{L}^{p'}(0,T; \cic{L}^{z(p)}(\Omega))$. Using \eqref{anni}, we see that
$X_{p'}
$
is the $L^{p}(0,T; \cic{L}^{s(p)}(\Omega))$-closure of  $(\cic{H}^{z(p)}_\tau)^\perp\otimes \D(0,T)$,
that is,
$
X_{p'} = L^{p}(0,T; \nabla W^{1,s(p)}(\Omega))$. Therefore, we proved that
\begin{equation} \label{gradp2}
\exists \pi \in L^{p}(0,T;  W^{1,s(p)}(\Omega)): \cic{w}=-\nabla \pi.
\end{equation}
Summarizing \eqref{Sol-reg}, \eqref{udotgrad}, \eqref{gradp}, \eqref{gradp2}, Theorem \ref{ThEx1} is proven.
\end{proof}

We now come to the second main theorem, which deals with \eqref{P} on the rectangle $\Omega=[0,L_1]\times[0,L_2]$ in the  range of exponents $2<p\leq \infty$ not covered by Theorem \ref{ThEx1}. In particular, this theorem contains the uniqueness of weak solutions to \eqref{P} on a rectangle with bounded initial vorticity.
\begin{theorem} \label{ThEx3} Let $\Omega=[0,L_1] \times [0,L_2]$ and   $2<p\leq \infty$, and  $$\u_0 \in \cic{V}^{1,p}(\Omega),\qquad \f \in  L^{p}\big(0,T; \cic{W}^{1,p}(\Omega)\big)  $$ be given.
\begin{itemize}\item[$(a)$] If $p<\infty$, there exists a pair $(\u,\pi)$ which is an \emph{(S)}$_p$-weak solution to the Euler system \eqref{P} with data $\u_0,\f$.
\item[$(b)$] If $p=\infty$, there exists a unique (up to an additive constant in $\pi$) pair $ (\u ,\pi )$ which is an \emph{(S)}$_q$-weak solution to the Euler system \eqref{P} with data $\u_0,\f$ for all $2<q<\infty$, with the additional regularity
$$
\cic{u}  \in L^{\infty}\big(0,T; \cic{V}^{1,\infty} (\Omega)\big).
$$
\end{itemize}
\end{theorem}\begin{proof}[Proof of  the existence in Theorem \ref{ThEx3}] Let for the moment $2<p<\infty$ and $\u_0,\f$ be given as in the statement of the theorem. 
We indicate again with $\u^{(n)}$ the solution to the approximating problems \eqref{Pn} on the subdomains $\Omega_n$, with initial and forcing data $\u^{(n)}_0, \f^{(n)}$ defined by the same procedure used in Subsection  \ref{ss41p}. Set also  $\omega^{(n)}=\curl \u^{(n)}$.  Repeating word by word the proofs of Lemmata  \ref{apvort} and \ref{eqvort} yields that $
\{\widetilde{ \omega^{(n)}}\}$ is equicontinuous in $\C([0,T]; \mathrm{w}-L^{p}(\Omega))$ and obeys the bound 
\begin{equation}
\label{vortpf21}
\|\widetilde{ \omega^{(n)}}\|_{L^\infty(0,T; L^{p}(\Omega))} \leq \e^{T}\big( \|\u_0\|_{\cic{V}^{1,p}(\Omega)} + \|\cic{f}\|_{L^p((0,T);\cic{W}^{1,p}(\Omega))}\big)^{\frac{1}{p}}:= \e^{T} R_p. 
\end{equation}
Letting  $\omega$ be a limit point of $
\{\widetilde{ \omega^{(n)}}\}$ in $\C([0,T]; \mathrm{w}-L^{p}(\Omega))$ and defining the candidate solution $\u:= \mathbf{K}_\Omega \omega$, and the stream function $\Psi= \mathrm{G}_\Omega\omega$, it follows from the elliptic regularity estimate \eqref{embedding} that
\begin{equation}
\label{boundPsi2}
\Psi(t) \in W^{2,p} \cap W^{1,p}_0(\Omega), \; \|\Psi(t)\|_{ W^{2,p}(\Omega) } \lesssim \|\omega\|_{L^\infty(0,T; L^p(\Omega))}   \leq \Q(R_p), \quad\forall t \in [0,T].
\end{equation} 
We instead repeat the proof of Lemmata \ref{apder} in the case $p=2$, and obtain for the sequence of stream functions $\Psi^{(n)}= \mathrm{G}_{\Omega_n}\omega^{(n)} $ is such that  \begin{equation} \label{vortpf21bis} \widetilde{  \Psi^{(n)}} \to \Psi \textrm{    in }
   L^{2}(0,T; H^1_0 (\Omega) ),\qquad \{\widetilde{\pt \Psi^{(n)}}\}  \; \textrm{    bdd   in }
   L^{2}(0,T; L^  {s }(\Omega) ), \quad 1<s<2,
 \end{equation}  
 so that also $\pt \Psi \in L^{2}(0,T; L^  {s }(\Omega) ) $ for each  $1<s<2$.
The convergence \eqref{vortpf21bis} allows us to repeat the  arguments of Subsection \ref{ss43p}, showing that the candidate solution $\u:= \mathbf{K}_\Omega \omega$ is indeed a weak solution in the sense of Subsection \ref{DefWS} (with $p=2$), as well as the arguments of the proof of Theorem \ref{ThEx1} for the case $p=2$, recovering the pressure $\pi$ and showing also that $(\u,\pi)$ is a (S)$_2$-weak solution  in the sense of 
Definition \ref{Defsws}, with the additional property that
\begin{equation} \label{bddfine}
\u \in  L^\infty(0,T; \cic{V}^{1,p}(\Omega)) \hookrightarrow L^\infty(0,T; \cic{W}^{1,p}(\Omega)), \qquad \u(t) \in \cic{W}^{1,p}(\Omega) \quad\forall t \in [0,T],
\end{equation}
a consequence of \eqref{boundPsi2}. To upgrade to an (S)$_p$-weak solution, we are only left to prove that $\pt u, \nabla \pi \in L^p(0,T; \cic{L}^{p}(\Omega))$. It is easy to see that
$$
\pt \u =\pt (P_\Omega \u) =P_\Omega (\pt \u) = - P_\Omega \big( (\u \cdot \nabla) \cic u + \f\big)
$$
equalities holding almost everywhere in $\Omega \times (0,T)$. Using that $P_\Omega$ is a bounded operator on $\cic{L}^p(\Omega)$ one can  estimate (here the implied constants might depend on $p>2$)
\begin{align*}
\|\pt \u\|_{L^p(0,T;\cic{L}^p(\Omega))} & \lesssim \|    (\u \cdot \nabla) \u\|_{L^p(0,T;\cic{L}^p(\Omega))}  + \|\f \|_{L^p(0,T;\cic{L}^p(\Omega))}  \\ &   \leq  \|\nabla\u\|_{{L^p(0,T;\cic{L}^p(\Omega))}} \|\u\|_{{L^\infty(0,T;\cic{L}^\infty(\Omega))}} + R_p \\ & \lesssim  \|\u\|_{{L^\infty(0,T;\cic{W}^{1,p}(\Omega))}}^2 +  R_p<\infty;
\end{align*}
we have also used the Sobolev embedding $\cic{W}^{1,p}(\Omega)  \hookrightarrow\cic{L}^\infty(\Omega)$ and \eqref{bddfine} to conclude. Arguing exactly like in the proof of Theorem \ref{ThEx1}, one then also recovers that $\nabla \pi \in L^p(0,T; \cic{L}^{p}(\Omega))$, which was what was left to prove.

 The case $p=\infty$ can be dealt with as follows. One constructs an approximating sequence of problems as in Subsection \ref{ss41p}, with data
\begin{equation}
\label{approxdatas}
\widetilde{\omega^{(n)}_0} \to \omega_0  \quad \textrm{in } \; L^{\infty}(\Omega), \qquad \|\f^{(n)}-\f\|_{ L^\infty(0,T; \cic{W}^{1,\infty}(\Omega_n))} \to 0, \quad n \to \infty.
\end{equation}
Since the  $R_p$ in \eqref{vortpf21} are bounded uniformly in $2<p<\infty$ by 
$$
R:=R(\u_0,\f):=\big(\|\u_0\|_{{\cic{V}^{1,\infty}(\Omega)}}+\|{\cic{f}}\|_{ L^\infty(0,T; {\cic{W}^{1,\infty}(\Omega))}}\big)
$$
one can pass to the $\lim\sup$ as $p \to \infty$, and obtain that 
\begin{equation}
\label{vortpf24}
\|\widetilde {\omega^{(n)}}\|_{L^\infty(0,T; L^{\infty}(\Omega))} \leq  \e^T R.
\end{equation}
Arguing similarly to what we did in Lemma \ref{eqvort} will   give that the sequence $\{\widetilde {\omega^{(n)}}\}$  is precompact in $\C([0,T]; \mathrm{w}^\star-L^{\infty}(\Omega))$. 
One then chooses once again a limit point  $\omega$ of $
\{\widetilde {\omega^{(n)}}\}$ in $\C([0,T]; \mathrm{w}^\star-L^{\infty}(\Omega))$ and defines the candidate solution $\u:= \mathbf{K}_\Omega \omega$. An application of Proposition \ref{bmolaw} then entails   
$ \u \in L^\infty(0,T; \cic{V}^{1,\infty}(\Omega)) \hookrightarrow L^\infty(0,T; \cic{W}^{1,\mathrm{bmo}_z}(\Omega));$ in particular,
\begin{equation} \label{bs2}
\|\u\|_{{L^\infty(0,T;  \cic{W}^{1,\mathrm{bmo}_z}(\Omega))}} \leq  C R(\u_0,\cic{f})\e^T.\end{equation}
Arguing in the same way as in the case $2<p<\infty$, one obtains that $\u$   is a (S)$_p$-weak solution  in the sense of 
Definition \ref{Defsws} for all $2<p<\infty$; this completes the proof of existence in the case $p=\infty$. 
\end{proof} 
\begin{remark}    Note that in the case $p=\infty$ we do not have a uniform $L^\infty(0,T;\cic{bmo}_{r}(\Omega ))$ estimate on $\pt   \u $. Despite $(\u \cdot \nabla) \cic u$ being indeed bounded in $L^\infty(0,T; \cic{bmo}_{r}(\Omega))$, to achieve such an estimate, we would need to prove first that $P_{\Omega} \in \mathcal{L}( \cic{bmo}_{z}(\Omega ) \to \cic{bmo}_{r}(\Omega )  )$, analogues of which hold true for smoother domains.\end{remark}

\begin{proof}[Proof of  the uniqueness in Theorem \ref{ThEx3}] The proof is analogous to Yudovich's proof in the case of a smooth domain. With $K$ we denote a positive constant, possibly varying from line to line and depending only on $(\u_0,\cic{f})$.   Assume that there exist two solutions $(\u^{1},\pi^1), (\u^{2},\pi^2)$ corresponding to the same data $(\u_0, \cic{f})$. Preliminarily observe that, from \eqref{JN}, Proposition \ref{bmolaw}, and \eqref{bs2}, 
\begin{align}
\label{bound}
\|\nabla \u^j\|_{\cic{L}^p(\Omega)}& \leq  C_\Omega p \|\nabla \u^j\|_{\mathrm{bmo}_z(\Omega)^{2\times 2}} \leq C_\Omega p \|\curl \u^j\|_{\mathrm{bmo}_z(\Omega) } \\&  \leq C_\Omega p \|\curl \u^j\|_{L^\infty(\Omega) } \leq C_\Omega       R(\u_0,\cic{f})p:= K p  \nonumber
\end{align}
for all $2 \leq p <\infty$.
  The difference $\u=\u^2-\u^1, \, \pi = \pi^2-\pi^1$ then satisfies the equation
\begin{equation}
\label{uniqueness}
\pt \u + (\u \cdot \nabla) \u^1  + (\u^2 \cdot \nabla) \u +\nabla \pi =0
\end{equation}
almost everywhere in $\Omega \times (0,T)$. Denote by $Y(t):= \|\u(t)\|_{\cic{L}^2(\Omega)}^2$. We take the scalar product of \eqref{uniqueness} with $\u$ and integrate on $\Omega$. After (legitimately) integrating by parts, applying H\"older's inequality with an arbitrary $p>2$ to be chosen later, and using interpolation one obtains the differential inequality
\begin{align}
\label{uniquenessaaa}
&\quad \ddt Y = - \big((\u \cdot \nabla) \u^1, \u\big)_{\Omega}   \leq K\|\nabla \u^1\|_{L^p(\Omega)} \|\u\|_{\cic{L}^{2p'}(\Omega)} \\ & \leq K\|\nabla \u^1\|_{L^p(\Omega)} \|\u\|_{\cic{L}^{2}(\Omega)}^{\frac{1}{p'}} \|\u\|_{\cic{L}^{\infty}(\Omega)}^{\frac{1}{p}}    \leq KpY^{1-\frac1p};\nonumber\end{align}
 in the last step, we used \eqref{bs2} to bound $\|\u\|_{\cic{L}^{\infty}(\Omega)}$, and \eqref{bound}. The bound $Y(t) \leq K$ for all $t \in [0,T]$ makes possible to choose a constant $M$ large  enough, depending only on $(\u_0,\cic{f})$ and $\Omega$,  such that $p=p(t) =\log \frac{M}{Y(t)} >2 $ for all $t$. Choosing $p=p(t)$ in \eqref{uniquenessaaa}, and observing that with this choice $Y^{-\frac{1}{p}}\leq K$,
 the above differential inequality then turns into 
 $$
 \ddt Y(t) \leq KY(t) \log \frac{M}{Y(t)}.
 $$
 Let $\eps>0$ be given; integrating on $(\eps,t)$, we obtain that
\begin{equation}
\label{uniqfin}
 Y(t) \leq M \big(Y(\eps)/M)^{\e^{-K(t-\eps)}}, \qquad \forall t >\eps. 
\end{equation}
 Due to the continuity $\u^i \in \C\big([0,T]; \cic{L}_\tau^2 (\Omega)\big),$ the functional $Y$ is continuous on $[0,T]$ and $Y(0)=0$. Passing to the limit  as $\eps \to 0$ in \eqref{uniqfin} we obtain that $Y(t) \equiv 0 $ on $[0,T]$. Thus $\u^1=\u^2$, which  forces $\nabla\pi^1=\nabla\pi^2$ too. This completes the proof. \end{proof}
\begin{remark} \label{remLL} We further comment on the spatial regularity of the unique solution  $\u$ of Theorem \ref{ThEx3} in the case $\u_0\in \cic{V}^{1,\infty}(\Omega)$. Inequality \eqref{bound} holding for each $2\leq p<\infty$ implies, via standard arguments (see for example \cite[Lemma 4.28]{Adams}), that $\u=\u^j$ is a log-Lipschitz vector field on $\Omega$, that is
$$
\|\u\|_{\cic{LL}(\Omega)}:= \|\u\|_{\cic{L}^\infty(\Omega)}+\sup_{\substack{x,y \in \Omega \\ 0<|x-y|\leq 1}} \frac {|\u(x)-\u(y)|}{\psi_{LL}(|x-y|)  } \lesssim K, \qquad \psi_{LL}(s) :=  s{\log(\e+s^{-1})}.
$$
This is sufficient to ensure that the flow 
$
(x_0,t) \mapsto x(t;x_0):= \Phi_t(x_0)
$
generated by the Cauchy problem
$$
\begin{cases}
\dot x (t;x_0) = \u(x (t;x_0), t) & t>0,
\\ x (t;x_0)= x_0 \in \Omega
\end{cases}
$$
is uniquely defined. It is easy to check  that the vorticity $\omega:= \curl \u$ then satisfies
$$
\omega(x,t):=\curl\,\u_0\big((\Phi_{t})^{-1}(x)\big) +  \int_0^t \curl \f \big( (\Phi_{t-s})^{-1}(x),s\big)\, \d s.
$$
\end{remark}

\section{Additional remarks}
\label{s6}
\subsection{(S)$_p$-weak solutions and turbulence} We show that, when $p\geq \frac32$, (S)$_p$-weak solutions  obey conservation of kinetic energy. In this range,  the embedding $$L^{\infty}(0,T; \cic{V}^{1,p} (\Omega))\hookrightarrow L^{p'}(0,T; \cic{L}^{z(p)}_\tau(\Omega)) $$
is continuous (see also Remark \ref{q432}). Therefore, we can multiply \eqref{PLq} by $\u$ and integrate on $\Omega \times (0,T)$. The convective term 
$b_\Omega(\u,\u,\u)$ vanishes identically, so that we get that the solution $\u$ of Theorem \ref{ThEx1} satisfies the energy equality
\begin{equation} \label{Eneq}
\ddt \|\u(t)\|_{\cic{L}^2(\Omega)}^2 - \big({\f}(t), \u(t) \big)_\Omega = 0, \qquad t \in (0,T). \end{equation}
 \label{remturb}
In relation with \eqref{Eneq}, we recall that the issue of conservation of energy for the solutions of the Euler equations is an important issue related to turbulence, although 
not in the focus of this article.  The relation of the conservation of energy for the Euler equations in relation with the Onsager conjecture \cite{Ons49} is well explained in 
e.g. the article by Schnirelman, \cite{Shn03} and its review in MathSciNet.  In this direction Uriel Frisch just draw our attention to the most recent result of De Lellis and 
Sz{\'e}kelyhidi  Jr., \cite{LS12} who proved the existence of solutions of the Euler equations which are Holder with exponent $< 1/10$ and which do not satisfy the 
conservation of energy; see also the earlier results \cite{CLS12}.  More precisely \cite{LS12} contains the construction of solutions to the Euler equations with initial vorticity in $BV(\Omega)$; the kinetic energy of these solutions decays with time, and a result of non uniqueness of solutions is proved (see Theorem \ref{ThEx3} below about uniqueness). See \cite{BT} for a broad discussion of those and other issues concerning the Euler equations and fluid mechanics more generally.

\subsection{(S)-Weak solutions  on $\mathbb T^2$} We consider the Euler system \eqref{P} on $\mathbb{T}^2=[0,1]^2$, replacing non-penetrating boundary conditions with periodic boundary conditions. We are interested in deriving a counterpart of Theorem \ref{ThEx3}, in the case $p=\infty$ in this setting.  The relevant space for the Euler data is then
$$
\cic{V}^{1,\infty}(\mathbb{T}^2)=\big\{ \u \in \cic{L}^\infty(\T): \curl \u \in L^\infty(\mathbb T^2), \,\div\u=0 \;\mathrm{ on }\; \mathbb T^2, \;\u \; \textrm{1-periodic in each } x_1,x_2\big\}
$$
In this setting, there is no need to smoothen up the domain, and one obtains smooth approximating solutions $\u^{(n)}$ originating from smooth  data $\u_0^{(n)},\cic{f}^{(n)}$ approximating respectively  $\u_0$ and  $\f$ as in \eqref{approxdatas}. The same  scheme of Sections \ref{s4} and \ref{s5} can be devised to obtain the following result.
\begin{theorem} Let  $\u_0 \in \cic{V}^{1,\infty}(\mathbb{T}^2), \; \f \in L^\infty(0,T; \cic{V}^{1,\infty}(\mathbb{T}^2)) $ be given. Then there exists a unique (S)-Weak solution to the Euler system \eqref{P} on $\mathbb{T}^2$, with periodic boundary condition, and data $\u_0, \f$; that is, a  pair $(\u:\Omega\times [0,T] \to \R^2,\pi:\Omega\times(0,T) \to \R)$ 
with $ \u(0)= \u_0 $ and 
\begin{align*}
& \cic{u} \in W^{1,\infty}\big([0,T];  \cic{BMO} (\mathbb{T}^2) \big)\cap L^{\infty}\big(0,T; \cic{V}^{1,\infty}(\mathbb{T}^2)\big), \quad \u(t) \in   \cic{V}^{1,\infty}(\mathbb{T}^2) \;\;\forall t \in[0,T],     \\
& \pt \u  \in L^\infty\big( 0,T  ; \cic{BMO}(\mathbb{T}^2) \big), \quad \nabla \u  \in L^\infty\big( 0,T ; \mathrm{BMO} (\mathbb{T}^2)^{2\times 2}\big)  \\
&  \nabla\pi \in  L^{q}\big( 0,T ; \cic{L}^{q}(\mathbb{T}^2 ), \qquad \forall q<\infty,  \\
&    
\pt \u + (\u \cdot \nabla) \u +\nabla \pi = {\f} \qquad \textrm{ in }  \; L^{q}\big(0,T;  \cic{L}^{q}(\mathbb{T}^2)\big)\;\;\forall q<\infty \quad\textrm{ and a.e. in }\mathbb{T}^2 \times (0,T) .  
\end{align*}
\end{theorem}
\begin{remark} Comparing the above Theorem to Theorem \ref{ThEx3}, the only improvement in regularity is that $\pt \u$ is uniformly in time bounded in  $\cic{BMO}(\mathbb{T}^2)$ (which, due to John-Nirenberg's inequality, is stronger than being uniformly bounded in  $\cic{L}^q(\Omega)$ for all $q<\infty$). This descends from the uniform boundedness of $\pt \u^{(n)}$, which is obtained in turn by modifying the proof of Lemma \ref{apder} as follows. It is well known that the Leray projector $P_{\mathbb{T}^2}$ is given by
$$
(P_{\mathbb{T}^2} \v)_j = \v_j +R_j R_1 \v_1 +  R_j R_2 \v_2 
$$
where $R_j$ stands for the $j$-th Riesz transform, which is bounded on $\mathrm{BMO}(\T^2)$ by standard Calderon-Zygmund theory.  Therefore 
$$
\|P_{\mathbb{T}^2} \v\|_{\cic{BMO}(\T^2)},\;  \|P_{\mathbb{T}^2}^{\perp} \v\|_{\cic{BMO}(\T^2)} \lesssim \|  \v\|_{\cic{BMO}(\T^2)}.
$$
Arguing as in \eqref{fill}, we then have, for each fixed $t \in (0,T)$, $$
 \|P_{\mathbb{T}^2}^{\perp} (\omega^{(n)} \u^{(n)})\|_{\cic{BMO}(\T^2)} \lesssim  \|  (\omega^{(n)} \u^{(n)})\|_{\cic{BMO}(\T^2)} \leq \| \omega^{(n)} \u^{(n)})\|_{\cic{L}^\infty(\T^2)} \leq \Q(\u_0, \f).  $$ Carrying on with  the proof of Lemma \ref{apder} with no further changes implies the uniform bound
$$
 \|\pt \u^{(n)}\|_{L^\infty(0,T;\cic{BMO}(\T^2))} \lesssim  \|  (\omega^{(n)} \u^{(n)} \|_{\cic{BMO}(\T^2)} \leq \| \omega^{(n)} \u^{(n)} \|_{\cic{L}^\infty(\T^2)} \leq \Q(\u_0, \f)$$
as claimed.
\end{remark}
\subsection{Theorem \ref{ThEx3} for more general domains with corners and open problems} We wish to discuss a    more general class of domains,  of which the rectangle $[0,L_1]\times [0,L_2]$, object of Theorem \ref{ThEx3}, is possibly the simplest instance. We say that  $\Omega \subset \R^2$ is a \emph{polygonal-like domain} if it is a bounded  simply connected open set and 
$\partial \Omega$ is a piecewise $\C^2$ planar curve, with finitely many points $\{S_1,\ldots,S_N\}$ of discontinuity for the tangent vector $\mathbf{n}(x),$ 
$x \in \partial \Omega$; we call $\Gamma_j$ the component of $\partial\Omega$ with end points $S_j,S_{j+1}$ (and $S_{N+1} = S_1 $ ). 
Then the tangent and normal vectors   $\cic{t}_j,\mathbf{n}_j$ to $\Gamma_j$  belong to $   \C^{1}(\Gamma_j)$, for each $j=1,\ldots,N$, 
and we write
$$
\mathbf{n}_{-} (S_j)= \lim_{\substack{x \to S_j \\x \in \Gamma_{j-1}}  } \mathbf{n} _j(x), \qquad \mathbf{n}_{+} (S_j)= 
\lim_{\substack{x \to S_j \\x \in \Gamma_{j }}  } \mathbf{n} _j(x).
$$
We assume for definiteness that $\mathbf{n}_{-} (S_j) \neq \mathbf{n}_{+} (S_j)$
and set
\begin{equation}
  \alpha_j := \mathrm{angle}(\mathbf{n}_{-} (S_j), \mathbf{n}_{+} (S_j)), \quad j=1,\ldots,N, \qquad \underline{\alpha}:=\min \{\alpha_j\}, \overline{\alpha}:= \max \{\alpha_j\}.
\end{equation}  
For this class of domains, under the additional assumption that $\overline \alpha \leq \frac\pi2$, given  $1<p<\infty $ and $f \in L^p(\Omega)$,  the unique solution $\mathrm{G}_\Omega f \in H^{1}_0(\Omega)$ belongs to $W^{2,p}(\Omega)$ as well and satisfies the estimate
\begin{equation}
\label{epestpoly}
 \|{{\mathrm{G}}}_{\Omega}f\|_{W^{2,p}(\Omega)} \leq C_{p,\Omega} \|f\|_{L^p(\Omega)}.
\end{equation}
If it is the case that $\frac\pi2<  \alpha_j <\pi$ for some $j$, a singular part is present and analogues of \eqref{epestpoly} only hold in a range of exponents $1<p <p_0, $ with $p_0 >2$ depending on $\overline \alpha$. These results are taken from Sections 4.1-4.4 of Grisvard's classical treatise \cite{Grisvard}   (in particular, Theorem 4.4.4.13 therein).  See also \cite{Grisvard2}.
\label{remarkpoly}
Tracking the proof in \cite{Grisvard}, one finds that dependence on $p$ of the constant $C_{p,\Omega}$ in \eqref{epestpoly} as $p \to \infty$ is $O(p^2)$, in contrast to the case of smooth domains, where $C_{p,\Omega}=O(p)$. The appearance of an extra power of $p$ is intrinsic in Grisvard's proof, which reduces the study of the Dirichlet problem on the corner with aperture $\alpha$ to an elliptic problem on the strip $\R\times (0,\alpha)$. The (second derivatives) of the correspondent Green's functions on the strip  are  Calder\'on-Zygmund kernels only separately in each variable, and thus fall into the scope of product Calder\`on-Zygmund theory. In general,   the $L^p$ norm of a product Calderon-Zygmund operator  is allowed to grow at order $p^2$ as $p\to\infty$, and the relevant endpoint space is no longer $\mathrm{BMO} $, but the smaller    product $\mathrm{BMO} (\R\times (0,\alpha))$. See \cite{CF} for references on this point.

A consequence of our analysis in Section \ref{Sec3}, in particular, of Proposition \ref{bmolaw}, is that $C_{p,\Omega}=O(p)$ in the particular case of a rectangular domain: this is a consequence of \eqref{epest3} and John-Nirenberg's inequality. 
Our analysis does not extend to a general polygonal-like domain with acute corners (where an estimate of the type \eqref{epestpoly}  for all $p>2$ is allowed to hold).
However, estimate \eqref{epest3} can be shown to hold also for polygonal-like domains  where all  angles are of the type $\alpha_j=\frac{\pi}{m(j)}$ for some integer $m(j)\geq 2$. The proof is analogous: after localization  and a suitable change of coordinates which straightens up the boundary, one studies the resulting (general) elliptic problem in the  corner with aperture $\alpha_j$, and applies an odd reflection argument to reach the halfspace. The proof is finished by applying the analogue of Theorem \ref{stein-thm} for a general elliptic operator of class $\C^2$. Thus, for a polygonal-like domain with angles    $\alpha_j=\frac{\pi}{m(j)}$, $C_{p,\Omega}=O(p)$ in  \eqref{epestpoly}. 

 As shown in the proof of Theorem \ref{ThEx3},  having $C_{p,\Omega}=O(p)$ in  \eqref{epestpoly} is sufficient to get    uniqueness  of solutions with bounded initial vorticity to the Euler system \eqref{P} on $\Omega$.
Hence, Theorem \ref{ThEx3} holds true \emph{verbatim} if we replace the rectangle $[0,L_1 ] \times[0,L_2]  $ with a domain of this sort. However, the proof of part (a) of Theorem \ref{ThEx3} only relies on the same techniques used for Theorem \ref{ThEx1} and on the elliptic estimate \eqref{epestpoly}. Therefore, this part of Theorem  \ref{ThEx3} holds  for all polygonal-like domains with $\overline\alpha \leq \frac\pi2$. It is an open question whether part (b) of Theorem \ref{ThEx3} holds for this class of domains as well.


\bibliography{BDT-Euler-Aug2012}{}
\bibliographystyle{amsplain}

\end{document}